\documentclass[11pt,reqno]{amsart}
\usepackage[body={7in,9in},centering]{geometry}
\usepackage{amsmath,amssymb,amsthm,mathrsfs}
\usepackage{color,xcolor}
\definecolor{cobalt}{RGB}{61,99,181}
\usepackage[colorlinks,citecolor=cobalt,linkcolor=cobalt]{hyperref}
\usepackage{exscale}
\usepackage{relsize}
\newtheorem{thm}{Theorem}[section]

\newtheorem{lem}[thm]{Lemma}
\newtheorem{rem}[thm]{Remark}

\numberwithin{equation}{section}

\date{\today}

\makeatletter

\newcommand{\Rmnum}[1]{\expandafter\@slowromancap\romannumeral #1@}
\makeatother
\allowdisplaybreaks[4]
\def\D{\mathbb{D}}
\def\T{\mathbb{T}}

\newcommand{\Z}{\mathbb{Z}}

\newcommand{\C}{\mathbb{C}}

\begin{document}
\title[Commuting Slant Toeplitz operators]{Commuting Slant Toeplitz operators on the Bergman space}

\author[Haiying Zhang]{Haiying Zhang}
\address{
\textsuperscript{1}
 College of Mathematics and Statistics, Chongqing University, Chongqing, 401331, P. R. China}
\email{zhy9299@126.com}

\date{\today}

\keywords{Toeplitz operator, Slant Toeplitz operator, Bergman space, commutativity}

\subjclass[2010]{47B35}

\begin{abstract}
This paper shows that on the Bergman space of the open unit disk, the  slant Toeplitz operator $T_{\overline{p}+\varphi}$ and $T_{\overline{p}+\psi}$ commute if and only if $\varphi=\psi$, where $\varphi$ and $\psi$ are both bounded analytic functions, and $p$ is an analytic polynomial.

\end{abstract}

\maketitle

\section{Introduction}

Let $\varphi(z)=\sum\limits_ {n=-\infty}^{\infty}\hat{\varphi}_nz^n$ be a bounded measurable function on the unit circle $\T$, where $\hat{\varphi}_n=\left<\varphi,z^n\right>$ is the $n$-th Fourier series of $\varphi$ and $\{z^n\}_{n=-\infty}^{\infty}$ is the basis of $L^2(\T)$. The slant Toeplitz operator $A_\varphi$ in $L^2(\T)$ is given by the following matrix
$$
\left( \begin{array} {c c c c c c c} 
{{\ddots}} & {{}} & {{}} & {{}} & {{}} & {{}} & {{}} \\ 
{{\cdots}} & {{a_{-2}}} & {{a_{-3}}} & {{\cdots}} & {{}} & {{}} & {{}} \\ 
{{\cdots}} & {{a_{0}}} & {{a_{-1}}} & {{a_{-2}}} & {{\cdots}} & {{}} & {{}} \\ 
{{}} & {{\cdots}} & {{a_{1}}} & {{a_{0}}} & {{a_{-1}}} & {{}} & {{}} \\ 
{{}} & {{}} & {{\cdots}} & {{a_{2}}} & {{a_{1}}} & {{a_{0}}} & {{\cdots}} \\ 
{{}} & {{}} & {{}} & {{\cdots}} & {{a_{3}}} & {{a_{2}}} & {{\cdots}} \\ 
{{}} & {{}} & {{}} & {{}} & {{\ddots}} & {{}} & {{\ddots}} \\ 
\end{array} \right). 
$$
Let $W$ be the bounded linear operators on $L^2(\T)$ such that 
$$Wz^{2n}=z^n,\,Wz^{2n+1}=0,\,n\in \Z.$$
Mark C. Ho characterized the relation between $A_\varphi$ and $W$, that is $A_{\varphi}=WM_{\varphi}$, where $M_\varphi$ is the multiplication operator.

The Bergman space $L_a^2$ is the closed subspace of $L^2(\D,dA)$ consisting of analytic functions on $\D$, so the Bergman space $L_a^2$ on the unit disk $\D=\{z\in \C:|z|<1\}$ is defined as 
$$L_a^2(\D)=\Big\{f~\text{is analytic on }\D : \|f\|^2=\int_{\D}|f(z)|^2dA<\infty\Big\},$$
where $dA$ denotes the normalized area measure. 
The Bergman space $L_a^2$ possesses a natural orthonormal basis $\{1,\sqrt{2}z,\cdots,\sqrt{n+1}z^n,\cdots\}.$
For a function 
$$f(z)=\sum_{n=0}^{\infty}a_nz^n\in L_a^2,$$
its norm is given by
$$\|f\|^2=\sum_{n\geq 0}\frac{|a_n|^2}{n+1}.$$

For $\varphi \in L^\infty(\D,dA)$, we define the slant Toeplitz operator $B_\varphi$ with symbol $\varphi$ on the Bergman space as 
$$B_{\varphi}f=WT_{\varphi}f,$$
where $W$ on $L_a^2$ is defined by 
$$Wz^{2n}=z^n,\,Wz^{2n+1}=0,\,n\geq0.$$
And $T_\varphi f=P(\varphi f)$ is the Toeplitz operator on the Bergman space $L_a^2$, in which $P$ is the orthogonal projection from $L^2(\D,dA)$ onto $L_a^2$.

In recent years, Brown and Halmos \cite{[1]} described the commuting Toeplitz operators on the Hardy space of the unit disk, while Axler and Čučković \cite{[2]} characterized the commuting Toeplitz operators with harmonic symbols on the Bergman space of the unit disk in 1991. The resulting Brown-Halmos type problems have since remained a central research focus in operator theory. Over the past two decades, beyond classical Toeplitz operators, researchers have extended the investigation of Brown-Halmos type theorems to various other operator classes. In \cite{[3]}, Singh and Gupta introduced the concepts of slant Toeplitz operators and $k$-th order slant Toeplitz operators on Fock spaces, derived their Berezin transforms, and established Brown-Halmos type theorems for $k$-th order slant Toeplitz operators. Liu Chaomei and Lu Yufeng \cite{[4]} studied the products and commutativity of slant Toeplitz operators on the Bergman space $L_a^2$, establishing Brown-Halmos type theorems for $k$-th order slant Toeplitz operators with harmonic polynomial, analytic, and conjugate analytic symbols in \cite{[5]}. Additionally, they proved the equivalence of commutativity and essential commutativity for slant Toeplitz operators in \cite{[6]}, and investigated algebraic, spectral, and commutativity properties of $ k$-th order slant Toeplitz operators on $L_a^2$. In \cite{[7]}, Singh and Sharma investigated generalized slant Toeplitz operators with symbols in the derivative Hardy space $S^{2}(\mathbb{D})$. Subsequently, Datt and Pandey \cite{[8]} extended the study of commutativity and essential commutativity relations for $k$-th order multivariate slant Toeplitz operators on $L^{2}(\mathbb{T}^{n})$. Furthermore, Datta \cite{[9]} generalized these results to $k$-th order weighted slant Toeplitz operators on weighted sequence spaces.

These developments motivate our further investigation into the commutativity of slant Toeplitz operators with harmonic symbols on the Bergman space. Notably, slant Toeplitz operators with harmonic symbols on the Bergman space exhibit structural parallels to their counterparts on the Hardy space. Recalling that the harmonic extension \cite{[10]} provides a natural correspondence between bounded functions on the unit circle and bounded harmonic functions on the unit disk, slant Toeplitz operators with harmonic symbols on the Bergman space inherit analogous properties to those on the Hardy space. This observation drives our systematic study of commutativity for slant Toeplitz operators with harmonic symbols in the Bergman space framework.

\section{Preliminary}
\begin{lem}
$W^*Wz^{2n}=W^*z^n=\frac{2n+1}{n+1}z^{2n}$ for $n=0,1,2,\cdots.$
\end{lem}
\begin{lem}
For $f\in L_a^2(\D)$, $W^*z^n=\frac{2n+1}{n+1}z^{2n} (n=0,1,2,\cdots)$ and $\|f\|_2\leq \|W^*f\|_2\leq \sqrt{2}\|f\|_2$.
\end{lem}
\begin{lem}
If $\varphi \in H^{\infty}(\D)$, then $T_\varphi W=WT_{\varphi(z^2)}$.
\end{lem}

\section{Commuting Slant Toeplitz operators on the Bergman space}
\begin{lem}\label{a}
Let $f=a\overline{z}+\varphi,\,g=b\overline{z}+\psi$, $\varphi$ and  $\psi$ be two bounded analytic functions on $\mathbb D$, $a,b\in\mathbb{R}$. There is a constant $c$ such that $B_fB_g=B_gB_f$ if and only if $f=cg$.
\end{lem}

\begin{proof}
Let us show the sufficiency first. Suppose that $f=cg$,  the result is trivial. 
To prove the necessity, we divide the proof into two steps.

\textbf{Step 1.}
We first consider the case where $a\neq 0$ and $b=1$, and assume that
$$\varphi(z)=\sum_{n=0}^{\infty}a_n z^n, \ \ \ \ \ \psi(z)=\sum_{m=0}^{\infty}b_m z^m.$$

Note the equality $B_fB_g=B_gB_f$ is equivalent to 
$$B_\varphi B_\psi z^{k}+aB_{\overline{z}}B_{\psi}z^{k}+ B_\varphi B_{\overline{z}} z^{k}= B_\psi B_\varphi z^{k}+aB_{\psi}B_{\overline{z}}z^{k}+B_{\overline{z}}B_{\varphi}z^{k}$$
for every  nonnegative integer $k$. Thus, we have
\begin{align*}
\begin{cases}
 B_\varphi B_\psi z^{k}=WT_\varphi W T_\psi z^{k}=WWT_{\varphi(z^2)}T_\psi z^{k}=WWT_{\varphi(z^2)\psi}z^{k},\\
 aB_{\overline{z}}B_\psi z^{k}=aW T_{\overline{z}}W T_\psi z^{k}, \\
 B_\varphi B_{\overline{z}}z^{k}=WT_\varphi W T_{\overline{z}}z^{k} =WWT_{\varphi(z^2)}T_{\overline{z}}z^{k}, \\
 B_\psi B_\varphi z^{k}=W T_\psi WT_\varphi z^{k}=WWT_{\psi(z^2)}T_\varphi z^{k}=WWT_{\psi(z^2)\varphi}z^{k},\\
 aB_\psi B_{\overline{z}} z^{k}=aWT_\psi W T_{\overline{z}}z^{k} =aWWT_{\psi(z^2)}T_{\overline{z}}z^{k}, \\
 B_{\overline{z}}B_\varphi z^{k}=WT_{\overline{z}} WT_{\varphi}z^k.
\end{cases}
\end{align*}
Since for every $k\geq0$, $B_\varphi B_\psi z^{k}, B_{\overline{z}}B_{\psi}z^{k}, B_\varphi B_{\overline{z}}z^{k}, B_\psi B_\varphi z^{k}, B_{\psi}B_{\overline{z}}z^{k}$ and $B_{\overline{z}}B_{\varphi}z^{k}$ are all analytic, letting 
$$B_\varphi B_\psi z^{2k}+aB_{\overline{z}}B_{\psi}z^{2k}+ B_\varphi B_{\overline{z}} z^{2k}:=\sum_{s=0}^{\infty}A_s z^s,$$ 
$$B_\psi B_\varphi z^{2k}+aB_{\psi}B_{\overline{z}}z^{2k}+B_{\overline{z}}B_{\varphi}z^{2k}:=\sum_{s=0}^{\infty}B_s z^s.$$
The definition of the operator $W$ implies that
$$WWz^{4s}=Wz^{2s}=z^s,\ \ \ \ \ s\geq0.$$
So we only need to consider the $4s$-th terms in $T_{\varphi(z^2)\psi}z^{2k},\ T_{\psi(z^2)\varphi}z^{2k},\ T_{\varphi(z^2)}T_{\overline{z}}z^{2k},\  T_{\psi(z^2)}T_{\overline{z}}z^{2k}$ and the $2s$-th terms in $T_{\overline{z}}W T_\varphi z^{2k},\ T_{\overline{z}}W T_\psi z^{2k}$.  Calculations show that
\begin{align*}
   T_{\varphi(z^2)\psi}z^{2k}=P\left(\varphi(z^2)\psi z^{2k}\right)=\sum_{n=0}^{\infty}\sum_{m=0}^{\infty}a_nb_mz^{2n+m+2k}:=\sum_{s=0}^{\infty}A_{1,s}z^s,
  \end{align*}
\begin{align*}
   aT_{\overline{z}}WT_\psi z^{2k} &=aT_{\overline{z}}W\left(\sum_{m=0}^{\infty}b_m z^m z^{2k}\right)\\
&=aT_{\overline{z}}\left(\sum_{m=0}^{\infty}b_{2m}z^{m+k}\right)\\
   &=aP\left(\overline{z}\sum_{m=0}^{\infty}b_{2m}z^{m+k}\right)\\
   &=a\sum_{m=0}^{\infty}\frac{m+k}{m+k+1}b_{2m}z^{m+k-1}\\
   &:=\sum_{s=0}^{\infty}A_{2,s}z^s,
  \end{align*}
 \begin{align*} 
 T_{\varphi(z^2)}T_{\overline{z}}z^{2k} &=T_{\varphi(z^2)}P(\overline{z}z^{2k})\\
 &=T_{\varphi(z^2)}\left(\frac{2k}{2k+1}z^{2k-1}\right)\\
 &=P\left(\sum_{n=0}^{\infty}\frac{2k}{2k+1}a_{n}z^{2n}z^{2k-1}\right)\\
&=\sum_{n=0}^{\infty}\frac{2k}{2k+1}a_{n}z^{2n+2k-1}\\
&:=\sum_{s=0}^{\infty}A_{3,s}z^s,
 \end{align*}
 \begin{align*}
 T_{\psi(z^2)\varphi}z^{2k}=P\left(\psi(z^2)\varphi z^{2k}\right)=\sum_{n=0}^{\infty}\sum_{m=0}^{\infty}a_nb_mz^{n+2m+2k}:=\sum_{s=0}^{\infty}B_{1,s}z^s; 
 \end{align*}
 \begin{align*}
 aT_{\psi(z^2)}T_{\overline{z}}z^{2k} &=aT_{\psi(z^2)}P(\overline{z}z^{2k})\\
 &=aT_{\psi(z^2)}\left(\frac{2k}{2k+1}z^{2k-1}\right)\\
 &=aP\left(\sum_{m=0}^{\infty}\frac{2k}{2k+1}b_{m}z^{2m}z^{2k-1}\right)\\      &=a\sum_{m=0}^{\infty}\frac{2k}{2k+1}b_{m}z^{2m+2k-1}\\
 &:=\sum_{s=0}^{\infty}B_{2,s}z^s, 
 \end{align*}
 \begin{align*} 
 T_{\overline{z}}W T_\varphi z^{2k} &=T_{\overline{z}}W\left(\sum_{n=0}^{\infty}a_n z^n z^{2k}\right)\\
 &=T_{\overline{z}}\left(\sum_{n=0}^{\infty}a_{2n}z^{n+k}\right)\\
 &=P\left(\overline{z}\sum_{n=0}^{\infty}a_{2n}z^{n+k}\right)\\
&=\sum_{n=0}^{\infty}\frac{n+k}{n+k+1}a_{2n}z^{n+k-1}\\
&:=\sum_{s=0}^{\infty}B_{3,s}z^s
 \end{align*}
 Using $B_\varphi B_\psi z^{k}+aB_{\overline{z}}B_{\psi}z^{k}+ B_\varphi B_{\overline{z}} z^{k}= B_\psi B_\varphi z^{k}+aB_{\psi}B_{\overline{z}}z^{k}+B_{\overline{z}}B_{\varphi}z^{k}$,  we obtain that $A_s=B_s$ for all $s\geq0$. Note that\\
$$WWT_{\varphi(z^2)}T_{\overline{z}}z^{2k}=WW\left(\sum_{n=0}^{\infty}\frac{2k}{2k+1}a_{n}z^{2n+2k-1}\right)=0,$$
$$WWT_{\psi(z^2)}T_{\overline{z}}z^{2k}=WW\left(\sum_{m=0}^{\infty}\frac{2k}{2k+1}b_{m}z^{2m+2k-1}\right)=0.$$\\
Thus, 
$$A_s=A_{1,4s}+A_{2,2s}=B_{1,4s}+B_{3,2s}=B_s.$$

Now we calculate the coefficients $A_s$ and $B_s$ defined above. Indeed, we have
\begin{align*}
  A_s&=A_{1,4s}+A_{2,2s}\\
     &=a_{2s-k}b_0+a_{2s-k-1}b_2+\cdots+a_0b_{4s-2k}+\frac{2s+1}{2s+2}ab_{4s-2k+2}\\
\end{align*}
and
\begin{align*}
  B_s&=B_{1,4s}+B_{3,2s}\\
     &=a_{0}b_{2s-k}+a_{2}b_{2s-k-1}+\cdots+a_{4s-2k}b_{0}+\frac{2s+1}{2s+2}a_{4s-2k+2}.
\end{align*}
This implies that
\begin{align}\label{5.1}
\begin{split}
  &a_{2s-k}b_0+a_{2s-k-1}b_2+\cdots+a_0b_{4s-2k}+\frac{2s+1}{2s+2}ab_{4s-2k+2}\\
  &=a_{0}b_{2s-k}+a_{2}b_{2s-k-1}+\cdots+a_{4s-2k}b_{0}+\frac{2s+1}{2s+2}a_{4s-2k+2}.
\end{split}
 \end{align}

Similarly, let 
$$B_\varphi B_\psi z^{2k+1}+aB_{\overline{z}}B_{\psi}z^{2k+1}+ B_\varphi B_{\overline{z}} z^{2k+1}:=\sum_{s=0}^{\infty}C_s z^s,$$
$$ B_\psi B_\varphi z^{2k+1}+aB_{\psi}B_{\overline{z}}z^{2k+1}+B_{\overline{z}}B_{\varphi}z^{2k+1}:=\sum_{s=0}^{\infty}D_s z^s.$$
We also need to consider $4s$-th terms in $T_{\varphi(z^2)\psi}z^{2k+1},\ T_{\psi(z^2)\varphi}z^{2k+1},\ T_{\varphi(z^2)}T_{\overline{z}}z^{2k+1},\  T_{\psi(z^2)}T_{\overline{z}}z^{2k+1}$ and the $2s$-th terms in $T_{\overline{z}}W T_\varphi z^{2k+1},\ T_{\overline{z}}W T_\psi z^{2k+1}$.  Calculations show that
\begin{align*}
T_{\varphi(z^2)\psi}z^{2k+1}=P\left(\varphi(z^2)\psi z^{2k+1}\right)=\sum_{n=0}^{\infty}\sum_{m=0}^{\infty}a_nb_mz^{2n+m+2k+1}:=\sum_{s=0}^{\infty}C_{1,s}z^s,
\end{align*}
\begin{align*}
aT_{\overline{z}}W T_\psi z^{2k+1} &=aT_{\overline{z}}W\left(\sum_{m=0}^{\infty}b_m z^m z^{2k+1}\right)\\
&=aT_{\overline{z}}\left(\sum_{m=0}^{\infty}b_{2m+1}z^{m+k+1}\right)\\
      &=aP\left(\overline{z}\sum_{m=0}^{\infty}b_{2m+1}z^{m+k+1}\right)\\
      &=a\sum_{m=0}^{\infty}\frac{m+k+1}{m+k+2}b_{2m+1}z^{m+k}\\
      &:=\sum_{s=0}^{\infty}C_{2,s}z^s,
\end{align*}
\begin{align*}
T_{\varphi(z^2)}T_{\overline{z}}z^{2k+1} &=T_{\varphi(z^2)}P(\overline{z}z^{2k+1})\\
&=T_{\varphi(z^2)}\left(\frac{2k+1}{2k+2}z^{2k}\right)\\
&=P\left(\sum_{n=0}^{\infty}\frac{2k+1}{2k+2}a_{n}z^{2n}z^{2k}\right)\\
      &=\sum_{n=0}^{\infty}\frac{2k+1}{2k+2}a_{n}z^{2n+2k}\\
      &:=\sum_{s=0}^{\infty}C_{3,s}z^s,
\end{align*}
\begin{align*}
T_{\psi(z^2)\varphi}z^{2k+1}=P\left(\psi(z^2)\varphi z^{2k+1}\right)=\sum_{n=0}^{\infty}\sum_{m=0}^{\infty}a_nb_mz^{n+2m+2k+1}:=\sum_{s=0}^{\infty}D_{1,s}z^s,
\end{align*}
\begin{align*}
aT_{\psi(z^2)}T_{\overline{z}}z^{2k+1} &=aT_{\psi(z^2)}P(\overline{z}z^{2k+1})\\
&=aT_{\psi(z^2)}\left(\frac{2k+1}{2k+2}z^{2k}\right)\\
&=aP\left(\sum_{m=0}^{\infty}\frac{2k+1}{2k+2}b_{m}z^{2m}z^{2k}\right)\\
      &=a\sum_{m=0}^{\infty}\frac{2k+1}{2k+2}b_{m}z^{2m+2k}\\
      &:=\sum_{s=0}^{\infty}D_{2,s}z^s,
\end{align*}
\begin{align*}
T_{\overline{z}}W T_\varphi z^{2k+1} &=T_{\overline{z}}W\left(\sum_{n=0}^{\infty}a_n z^n z^{2k+1}\right)\\
&=T_{\overline{z}}\left(\sum_{n=0}^{\infty}a_{2n+1}z^{n+k+1}\right)\\
&=P\left(\overline{z}\sum_{n=0}^{\infty}a_{2n+1}z^{n+k+1}\right)\\
      &=\sum_{n=0}^{\infty}\frac{n+k+1}{n+k+2}a_{2n+1}z^{n+k}\\
      &:=\sum_{s=0}^{\infty}D_{3,s}z^s.
\end{align*}
We also obtain that $C_s=D_s$ for all $s\geq0$. Therefore, 
$$C_s=C_{1,4s}+C_{2,2s}+C_{3,4s}=D_{1,4s}+D_{2,4s}+D_{3,2s}=D_s.$$
Thus, we concluded that
\begin{align*}
  C_s&=C_{1,4s}+C_{2,2s}+C_{3,4s}\\
     &=a_{2s-k-1}b_1+a_{2s-k-2}b_3+\cdots+a_0b_{4s-2k-1}+\frac{2s+1}{2s+2}ab_{4s-2k+1}+\frac{2k+1}{2k+2}a_{2s-k} \\
\end{align*}
and
\begin{align*}
  D_s&=D_{1,4s}+D_{2,4s}+D_{3,2s}\\
     &=a_{1}b_{2s-k-1}+a_{3}b_{2s-k-2}+\cdots+a_{4s-2k-1}b_{0}+\frac{2k+1}{2k+2}ab_{2s-k}+\frac{2s+1}{2s+2}a_{4s-2k+1}.
\end{align*}
This follows that
\begin{align}\label{5.2}
\begin{split}
&a_{2s-k-1}b_1+a_{2s-k-2}b_3+\cdots+a_0b_{4s-2k-1}+\frac{2s+1}{2s+2}ab_{4s-2k+1}+\frac{2k+1}{2k+2}a_{2s-k}\\
     &=a_{1}b_{2s-k-1}+a_{3}b_{2s-k-2}+\cdots+a_{4s-2k-1}b_{0}+\frac{2k+1}{2k+2}ab_{2s-k}+\frac{2s+1}{2s+2}a_{4s-2k+1}.
\end{split}
 \end{align}

Letting $s=t, k=0;\, s=t+1, k=2,$  respectively, in (\ref{5.1}), we get the following equations:
 \begin{align}\label{5.3}
 \begin{split}
  &a_{2t}b_0+a_{2t-1}b_2+\cdots+a_0b_{4t}+\frac{2t+1}{2t+2}ab_{4t+2}\\
  &=a_0b_{2t}+a_2b_{2t-1}+\cdots+a_{4t}b_0+\frac{2t+1}{2t+2}a_{4t+2};
  \end{split}
  \end{align}
\begin{align}\label{5.4}
\begin{split}
  &a_{2t}b_0+a_{2t-1}b_2+\cdots+a_0b_{4t}+\frac{2t+3}{2t+4}ab_{4t+2}\\
  &=a_0b_{2t}+a_2b_{2t-1}+\cdots+a_{4t}b_0+\frac{2t+3}{2t+4}a_{4t+2}.
  \end{split}
 \end{align} 
Subtracting Equation (\ref{5.4})  from Equation (\ref{5.3}) yields that
$$\frac{2}{(2t+2)(2t+4)}ab_{4t+2}=\frac{2}{(2t+2)(2t+4)}a_{4t+2}.$$
Since $a\neq 0$ and $\frac{2}{(2t+2)(2t+4)}\neq0$ for any $t\geq 0$, we obtain that 
\begin{align}\label{5.5}
\begin{split}
a_{4t+2}=ab_{4t+2}.
\end{split}
 \end{align}

Similarly, by setting $s=t, k=1;\, s=t+1, k=3$ in (\ref{5.1}), we obtain the following equations:
 \begin{align}\label{5.6}
 \begin{split}
  &a_{2t-1}b_0+a_{2t-2}b_2+\cdots+a_0b_{4t-2}+\frac{2t+1}{2t+2}ab_{4t}\\
  &=a_0b_{2t-1}+a_2b_{2t-2}+\cdots+a_{4t-2}b_0+\frac{2t+1}{2t+2}a_{4t};
 \end{split}
 \end{align}
 \begin{align}\label{5.7}
 \begin{split}
 &a_{2t-1}b_0+a_{2t-2}b_2+\cdots+a_0b_{4t-2}+\frac{2t+3}{2t+4}ab_{4t}\\
 &=a_0b_{2t-1}+a_2b_{2t-2}+\cdots+a_{4t-2}b_0+\frac{2t+3}{2t+4}a_{4t}.
 \end{split}
 \end{align}
Subtracting Equation (\ref{5.7})  from Equation (\ref{5.6}) yields that
$$\frac{2}{(2t+2)(2t+4)}ab_{4t}=\frac{2}{(2t+2)(2t+4)}a_{4t}.$$
Since $a\neq 0$ and $\frac{2}{(2t+2)(2t+4)}\neq0$ for any $t\geq 0$, we obtain that 
\begin{align}\label{5.8}
\begin{split}
a_{4t}=ab_{4t}.
\end{split}
 \end{align}
By (\ref{5.5}) and (\ref{5.8}), we have
\begin{align}\label{5.9}
\begin{split}
a_{2t}=ab_{2t}\ \ \mbox{for all}\  t\geq 0.
\end{split}
 \end{align}

Next, by setting $s=t+1,\,k=1;s=t+2,\,k=3; s=t+3,\,k=5$ in (\ref{5.2}), we obtain the following equations:
\begin{align}
\begin{split}
    & a_{2t}b_1+a_{2t-1}b_3+\cdots+a_0b_{4t+1}+\frac{2t+3}{2t+4}ab_{4t+3}+\frac{3}{4}a_{2t+1} \\
     &\ \ \ \ \ \ \ \ \ \ \ \ =a_1b_{2t}+a_3b_{2t-1}+\cdots+a_{4t+1}b_0+\frac{3}{4}ab_{2t+1}+\frac{2t+3}{2t+4}a_{4t+3};\label{5.10}
\end{split}
\end{align}
\begin{align}
\begin{split}
    & a_{2t}b_1+a_{2t-1}b_3+\cdots+a_0b_{4t+1}+\frac{2t+5}{2t+6}ab_{4t+3}+\frac{7}{8}a_{2t+1} \\
     &\ \ \ \ \ \ \ \ \ \ \ =a_1b_{2t}+a_3b_{2t-1}+\cdots+a_{4t+1}b_0+\frac{7}{8}ab_{2t+1}+\frac{2t+5}{2t+6}a_{4t+3};\label{5.11}
\end{split}
\end{align}
\begin{align}
\begin{split}
    & a_{2t}b_1+a_{2t-1}b_3+\cdots+a_0b_{4t+1}+\frac{2t+7}{2t+8}ab_{4t+3}+\frac{11}{12}a_{2t+1} \\
     &\ \ \ \ \ \ \ \ \ \ \ =a_1b_{2t}+a_3b_{2t-1}+\cdots+a_{4t+1}b_0+\frac{11}{12}ab_{2t+1}+\frac{2t+7}{2t+8}a_{4t+3}.\label{5.12}
\end{split}
\end{align}
Subtracting Equation (\ref{5.11})  from Equation (\ref{5.10}) and subtracting Equation (\ref{5.12})  from Equation (\ref{5.11}) yields
\begin{align*}
 \begin{cases}
\frac{2}{(2t+4)(2t+6)}(ab_{4t+3}-a_{4t+3})+\frac{1}{8}(a_{2t+1}-ab_{2t+1})=0,\\
\frac{2}{(2t+6)(2t+8)}(ab_{4t+3}-a_{4t+3})+\frac{1}{24}(a_{2t+1}-ab_{2t+1})=0.
 \end{cases}
\end{align*}
Denoting $ab_{4t+3}-a_{4t+3}$ as $M$ and $a_{2t+1}-ab_{2t+1}$ as $N$, respectively, we obtain that
\begin{align}\label{5.13}
 \begin{cases}
\frac{2}{(2t+4)(2t+6)}M+\frac{1}{8}N=0,\\
\frac{2}{(2t+6)(2t+8)}M+\frac{1}{24}N=0.
 \end{cases}
 \end{align}
The coefficient matrix of (\ref{5.13}) is
\[\left(\begin{array}{cc}
         \frac{2}{(2t+4)(2t+6)} & \frac{1}{8} \\
         \frac{2}{(2t+6)(2t+8)} & \frac{1}{24} 
\end{array} \right)\]
and it is invertible. Thus $M=N=0$, i.e., 
\begin{align}\label{5.14}
\begin{split}
a_{2t+1}=ab_{2t+1},\,a_{4t+3}=ab_{4t+3}
\end{split}
 \end{align}
for all $t\geq 0$. By (\ref{5.9}) and (\ref{5.14}) we obtain that $\varphi=a\psi$. Thus,  $f=ag$.

If $a=0$, $b=1$, we have $f=\varphi$, $g=\overline{z}+\psi$. This situation can be solved as described in \cite{[11]}.

\textbf{Step 2.}
If $b\neq1$ and $b\neq 0$, denoting $\frac{1}{b}\psi$ as $\phi$, let $h=\frac{1}{b}g=\overline{z}+\phi$. Then $f$ and $h$ satisfying the case in step 1.

Combining the first and second steps, we have completed the proof of this theorem.
\end{proof}
\begin{lem}\label{b}
Let $f=a{\overline{z}}^N+\varphi,\,g=b{\overline{z}}^N+\psi$, $\varphi$ and  $\psi$ be two bounded analytic functions on $\mathbb D$, $a,b\in\mathbb{R}$, $N\geq1$, where $N\in \Z$. There is a constant $c$ such that $B_fB_g=B_gB_f$ if and only if $f=cg$.
\end{lem}
\begin{proof}
Let us show the sufficiency first. Suppose that $f=cg$,  the result is trivial. 
To prove the necessity, we divide the proof into two steps.

\textbf{Step 1.}
We first consider the case where $b=1$ and assume that
$$\varphi(z)=\sum_{n=0}^{\infty}a_n z^n, \ \ \ \ \ \psi(z)=\sum_{m=0}^{\infty}b_m z^m.$$

If $N$ is odd, note that $B_fB_g=B_gB_f$ is equivalent to 
$$B_\varphi B_\psi z^{k}+aB_{{\overline{z}}^N}B_{\psi}z^{k}+ B_\varphi B_{{\overline{z}}^N} z^{k}= B_\psi B_\varphi z^{k}+aB_{\psi}B_{{\overline{z}}^N}z^{k}+B_{{\overline{z}}^N}B_{\varphi}z^{k}$$
for every  nonnegative integer $k$. Thus, we have
\begin{align*}
\begin{cases}
 B_\varphi B_\psi z^{k}=WT_\varphi W T_\psi z^{k}=WWT_{\varphi(z^2)}T_\psi z^{k}=WWT_{\varphi(z^2)\psi}z^{k},\\
 aB_{{\overline{z}}^N}B_\psi z^{k}=aW T_{{\overline{z}}^N}W T_\psi z^{k},\\
 B_\varphi B_{{\overline{z}}^N}z^{k}=WT_\varphi W T_{{\overline{z}}^N}z^{k} =WWT_{\varphi(z^2)}T_{{\overline{z}}^N}z^{k}, \\
 B_\psi B_\varphi z^{k}=W T_\psi WT_\varphi z^{k}=WWT_{\psi(z^2)}T_\varphi z^{k}=WWT_{\psi(z^2)\varphi}z^{k},\\
 aB_\psi B_{{\overline{z}}^N} z^{k}=aWT_\psi W T_{{\overline{z}}^N}z^{k} =aWWT_{\psi(z^2)}T_{{\overline{z}}^N}z^{k}, \\
 B_{{\overline{z}}^N}B_\varphi z^{k}=WT_{{\overline{z}}^N} WT_{\varphi}z^k.
\end{cases}
\end{align*}
Since for every $k\geq0$, $B_\varphi B_\psi z^{k}, B_{{\overline{z}}^N}B_{\psi}z^{k}, B_\varphi B_{{\overline{z}}^N}z^{k}, B_\psi B_\varphi z^{k}, B_{\psi}B_{{\overline{z}}^N}z^{k}$ and $B_{{\overline{z}}^N}B_{\varphi}z^{k}$ are all analytic, letting 
$$B_\varphi B_\psi z^{2k}+aB_{{\overline{z}}^N}B_{\psi}z^{2k}+ B_\varphi B_{{\overline{z}}^N} z^{2k}:=\sum_{s=0}^{\infty}A_s z^s,$$ 
$$B_\psi B_\varphi z^{2k}+aB_{\psi}B_{{\overline{z}}^N}z^{2k}+B_{{\overline{z}}^N}B_{\varphi}z^{2k}:=\sum_{s=0}^{\infty}B_s z^s.$$
Thus, we only need to consider the $4s$-th terms in $T_{\varphi(z^2)\psi}z^{2k},\ T_{\psi(z^2)\varphi}z^{2k},\ T_{\varphi(z^2)}T_{{\overline{z}}^N}z^{2k},\  T_{\psi(z^2)}T_{{\overline{z}}^N}z^{2k}$ and the $2s$-th terms in $T_{{\overline{z}}^N}W T_\varphi z^{2k},\ T_{{\overline{z}}^N}W T_\psi z^{2k}$.  Calculations show that
\begin{align*}
  T_{\varphi(z^2)\psi}z^{2k}=P\left(\varphi(z^2)\psi z^{2k}\right)=\sum_{n=0}^{\infty}\sum_{m=0}^{\infty}a_nb_mz^{2n+m+2k}:=\sum_{s=0}^{\infty}A_{1,s}z^s,
\end{align*}
\begin{align*}
  aT_{{\overline{z}}^N}W T_\psi z^{2k} &=aT_{{\overline{z}}^N}W\left(\sum_{m=0}^{\infty}b_m z^m z^{2k}\right)\\
  &=aT_{{\overline{z}}^N}\left(\sum_{m=0}^{\infty}b_{2m}z^{m+k}\right)\\
  &=aP\left({{\overline{z}}^N}\sum_{m=0}^{\infty}b_{2m}z^{m+k}\right)\\
      &=a\sum_{m=0}^{\infty}\frac{m+k+1-N}{m+k+1}b_{2m}z^{m+k-N}\\
      &:=\sum_{s=0}^{\infty}A_{2,s}z^s,
\end{align*}
\begin{align*}
T_{\varphi(z^2)}T_{{\overline{z}}^N}z^{2k} &=T_{\varphi(z^2)}P({{\overline{z}}^N}z^{2k})\\
&=T_{\varphi(z^2)}\left(\frac{2k+1-N}{2k+1}z^{2k-N}\right)\\
      &=P\left(\sum_{n=0}^{\infty}\frac{2k+1-N}{2k+1}a_{n}z^{2n}z^{2k-N}\right)\\
      &=\sum_{n=0}^{\infty}\frac{2k+1-N}{2k+1}a_{n}z^{2n+2k-N}\\
      &:=\sum_{s=0}^{\infty}A_{3,s}z^s,
\end{align*}
\begin{align*}
T_{\psi(z^2)\varphi}z^{2k}=P\left(\psi(z^2)\varphi z^{2k}\right)=\sum_{n=0}^{\infty}\sum_{m=0}^{\infty}a_nb_mz^{n+2m+2k}:=\sum_{s=0}^{\infty}B_{1,s}z^s,  
\end{align*}
\begin{align*}
aT_{\psi(z^2)}T_{{\overline{z}}^N}z^{2k} &=aT_{\psi(z^2)}P({{\overline{z}}^N}z^{2k})\\
&=aT_{\psi(z^2)}\left(\frac{2k+1-N}{2k+1}z^{2k-N}\right)\\
      &=aP\left(\sum_{m=0}^{\infty}\frac{2k+1-N}{2k+1}b_{m}z^{2m}z^{2k-N}\right)\\
      &=a\sum_{m=0}^{\infty}\frac{2k+1-N}{2k+1}b_{m}z^{2m+2k-N}\\
      &:=\sum_{s=0}^{\infty}B_{2,s}z^s,  
\end{align*}
\begin{align*}
T_{{\overline{z}}^N}W T_\varphi z^{2k} &=T_{{\overline{z}}^N}W\left(\sum_{n=0}^{\infty}a_n z^n z^{2k}\right)\\
&=T_{{\overline{z}}^N}\left(\sum_{n=0}^{\infty}a_{2n}z^{n+k}\right)\\
&=P\left({{\overline{z}}^N}\sum_{n=0}^{\infty}a_{2n}z^{n+k}\right)\\
      &=\sum_{n=0}^{\infty}\frac{n+k+1-N}{n+k+1}a_{2n}z^{n+k-N}\\
      &:=\sum_{s=0}^{\infty}B_{3,s}z^s.  
\end{align*}
Using $B_\varphi B_\psi z^{k}+aB_{{\overline{z}}^N}B_{\psi}z^{k}+ B_\varphi B_{{\overline{z}}^N} z^{k}= B_\psi B_\varphi z^{k}+aB_{\psi}B_{{\overline{z}}^N}z^{k}+B_{{\overline{z}}^N}B_{\varphi}z^{k}$,  we obtain that $A_s=B_s$ for all $s\geq 0$. Note that\\
$$WWT_{\varphi(z^2)}T_{{\overline{z}}^N}z^{2k}=WW\left(\sum_{n=0}^{\infty}\frac{2k+1-N}{2k+1}a_{n}z^{2n+2k-N}\right)=0,$$
$$WWT_{\psi(z^2)}T_{{\overline{z}}^N}z^{2k}=WW\left(\sum_{m=0}^{\infty}\frac{2k+1-N}{2k+1}b_{m}z^{2m+2k-N}\right)=0.$$\\
Thus, 
$$A_s=A_{1,4s}+A_{2,2s}=B_{1,4s}+B_{3,2s}=B_s.$$

Now we are going to calculate the coefficients $A_s$ and $B_s$ defined above. Indeed, we have
\begin{align*}
  A_s&=A_{1,4s}+A_{2,2s}\\
     &=a_{2s-k}b_0+a_{2s-k-1}b_2+\cdots+a_0b_{4s-2k}+\frac{2s+N}{2s+N+1}ab_{4s-2k+2N}\\
\end{align*}
and
\begin{align*}
  B_s&=B_{1,4s}+B_{3,2s}\\
     &=a_{0}b_{2s-k}+a_{2}b_{2s-k-1}+\cdots+a_{4s-2k}b_{0}+\frac{2s+N}{2s+N+1}a_{4s-2k+2N}.
\end{align*}
This implies that
\begin{align}\label{6.1}
\begin{split}
  &a_{2s-k}b_0+a_{2s-k-1}b_2+\cdots+a_0b_{4s-2k}+\frac{2s+N}{2s+N+1}ab_{4s-2k+2N}\\
  &=a_{0}b_{2s-k}+a_{2}b_{2s-k-1}+\cdots+a_{4s-2k}b_{0}+\frac{2s+N}{2s+N+1}a_{4s-2k+2N}.
\end{split}
 \end{align}

Similarly, we define 
$$B_\varphi B_\psi z^{2k+1}+aB_{{\overline{z}}^N}B_{\psi}z^{2k+1}+ B_\varphi B_{{\overline{z}}^N} z^{2k+1}:=\sum_{s=0}^{\infty}C_s z^s,$$
$$ B_\psi B_\varphi z^{2k+1}+aB_{\psi}B_{{\overline{z}}^N}z^{2k+1}+B_{{\overline{z}}^N}B_{\varphi}z^{2k+1}:=\sum_{s=0}^{\infty}D_s z^s.$$
We also need to consider $4s$-th terms in $T_{\varphi(z^2)\psi}z^{2k+1},\ T_{\psi(z^2)\varphi}z^{2k+1},\ T_{\varphi(z^2)}T_{{\overline{z}}^N}z^{2k+1},\  T_{\psi(z^2)}T_{{\overline{z}}^N}z^{2k+1}$ and the $2s$-th terms in $T_{{\overline{z}}^N}W T_\varphi z^{2k+1},\ T_{{\overline{z}}^N}W T_\psi z^{2k+1}$.  Calculations show that
\begin{align*}
T_{\varphi(z^2)\psi}z^{2k+1}=P\left(\varphi(z^2)\psi z^{2k+1}\right)=\sum_{n=0}^{\infty}\sum_{m=0}^{\infty}a_nb_mz^{2n+m+2k+1}:=\sum_{s=0}^{\infty}C_{1,s}z^s, 
\end{align*}
\begin{align*}
aT_{{\overline{z}}^N}W T_\psi z^{2k+1} &=aT_{{\overline{z}}^N}W\left(\sum_{m=0}^{\infty}b_m z^m z^{2k+1}\right)\\
&=aT_{{\overline{z}}^N}\left(\sum_{m=0}^{\infty}b_{2m+1}z^{m+k+1}\right)\\
      &=aP\left({{\overline{z}}^N}\sum_{m=0}^{\infty}b_{2m+1}z^{m+k+1}\right)\\
      &=a\sum_{m=0}^{\infty}\frac{m+k+2-N}{m+k+2}b_{2m+1}z^{m+k+1-N}\\
      &:=\sum_{s=0}^{\infty}C_{2,s}z^s, 
\end{align*}
\begin{align*}
T_{\varphi(z^2)}T_{{\overline{z}}^N}z^{2k+1} &=T_{\varphi(z^2)}P({{\overline{z}}^N}z^{2k+1})\\
&=T_{\varphi(z^2)}\left(\frac{2k+2-N}{2k+2}z^{2k+1-N}\right)\\
      &=P\left(\sum_{n=0}^{\infty}\frac{2k+2-N}{2k+2}a_{n}z^{2n}z^{2k+1-N}\right)\\
      &=\sum_{n=0}^{\infty}\frac{2k+2-N}{2k+2}a_{n}z^{2n+2k+1-N}\\
      &:=\sum_{s=0}^{\infty}C_{3,s}z^s, 
\end{align*}
\begin{align*}
T_{\psi(z^2)\varphi}z^{2k+1}=P\left(\psi(z^2)\varphi z^{2k+1}\right)=\sum_{n=0}^{\infty}\sum_{m=0}^{\infty}a_nb_mz^{n+2m+2k+1}:=\sum_{s=0}^{\infty}D_{1,s}z^s, 
\end{align*}
\begin{align*}
aT_{\psi(z^2)}T_{{\overline{z}}^N}z^{2k+1} &=aT_{\psi(z^2)}P({{\overline{z}}^N}z^{2k+1})\\
&=aT_{\psi(z^2)}\left(\frac{2k+2-N}{2k+2}z^{2k+1-N}\right)\\
      &=aP\left(\sum_{m=0}^{\infty}\frac{2k+2-N}{2k+2}b_{m}z^{2m}z^{2k+1-N}\right)\\
      &=a\sum_{m=0}^{\infty}\frac{2k+2-N}{2k+2}b_{m}z^{2m+2k+1-N}\\
      &:=\sum_{s=0}^{\infty}D_{2,s}z^s, 
\end{align*}
\begin{align*}
T_{{\overline{z}}^N}W T_\varphi z^{2k+1} &=T_{{\overline{z}}^N}W\left(\sum_{n=0}^{\infty}a_n z^n z^{2k+1}\right)\\
&=T_{{\overline{z}}^N}\left(\sum_{n=0}^{\infty}a_{2n+1}z^{n+k+1}\right)\\
      &=P\left({{\overline{z}}^N}\sum_{n=0}^{\infty}a_{2n+1}z^{n+k+1}\right)\\
      &=\sum_{n=0}^{\infty}\frac{n+k+2-N}{n+k+2}a_{2n+1}z^{n+k+1-N}\\
      &:=\sum_{s=0}^{\infty}D_{3,s}z^s. 
\end{align*}
We also obtain that $C_s=D_s$ for all $s\geq 0$. Therefore, 
$$C_s=C_{1,4s}+C_{2,2s}+C_{3,4s}=D_{1,4s}+D_{2,4s}+D_{3,2s}=D_s.$$
We concluded that
\begin{align*}
  C_s&=C_{1,4s}+C_{2,2s}+C_{3,4s}\\
     &=a_{2s-k-1}b_1+a_{2s-k-2}b_3+\cdots+a_0b_{4s-2k-1}+\frac{2s+N}{2s+N+1}ab_{4s-2k+2N-1}\\
     &\ \ \ \ +\frac{2k+2-N}{2k+2}a_{2s-k+\frac{N-1}{2}}
\end{align*}
and
\begin{align*}
  D_s&=D_{1,4s}+D_{2,4s}+D_{3,2s}\\
     &=a_{1}b_{2s-k-1}+a_{3}b_{2s-k-2}+\cdots+a_{4s-2k-1}b_{0}+\frac{2k+2-N}{2k+2}ab_{2s-k+\frac{N-1}{2}}\\
     &\ \ \ \ +\frac{2s+N}{2s+N+1}a_{4s-2k+2N-1}.
\end{align*}
This follows that
\begin{align}\label{6.2}
\begin{split}
&a_{2s-k-1}b_1+a_{2s-k-2}b_3+\cdots+a_0b_{4s-2k-1}+\frac{2s+N}{2s+N+1}ab_{4s-2k+2N-1}\\
&\ \ \ \ \ \ \ \ \ \ \ \ \ \ \ \ +\frac{2k+2-N}{2k+2}a_{2s-k+\frac{N-1}{2}}\\
     &=a_{1}b_{2s-k-1}+a_{3}b_{2s-k-2}+\cdots+a_{4s-2k-1}b_{0}+\frac{2k+2-N}{2k+2}ab_{2s-k+\frac{N-1}{2}}\\
     &\ \ \ \ \ \ \ \ \ \ \ \ \ \ \ \ +\frac{2s+N}{2s+N+1}a_{4s-2k+2N-1}.
\end{split}
 \end{align}

By setting $s=t, k=N;\, s=t+1, k=N+2$ in (\ref{6.1}), we obtain the following equations:
\begin{align}
\begin{split}
    & a_{2t-N}b_0+a_{2t-N-1}b_2+\cdots+a_0b_{4t-2N}+\frac{2t+N}{2t+N+1}ab_{4t} \\
     &\ \ \ \ \ \ \ \ \ \ \ \ =a_0b_{2t-N}+a_2b_{2t-N-1}+\cdots+a_{4t-2N}b_0+\frac{2t+N}{2t+N+1}a_{4t};\label{6.3}
\end{split}
\end{align}
\begin{align}
\begin{split}
    & a_{2t-N}b_0+a_{2t-N-1}b_2+\cdots+a_0b_{4t-2N}+\frac{2t+N+2}{2t+N+3}ab_{4t} \\
     &\ \ \ \ \ \ \ \ \ \ \ \ =a_0b_{2t-N}+a_2b_{2t-N-1}+\cdots+a_{4t-2N}b_0+\frac{2t+N+2}{2t+N+3}a_{4t}.\label{6.4}
\end{split}
\end{align}
Subtracting Equation (\ref{6.4})  from Equation (\ref{6.3}) yields that
$$\frac{2}{(2t+N+1)(2t+N+3)}ab_{4t}=\frac{2}{(2t+N+1)(2t+N+3)}a_{4t}.$$
Since $a\neq 0$ and $\frac{2}{(2t+N+1)(2t+N+3)}\neq0$ for any $t\geq 0$, we obtain that 
\begin{align}\label{6.5}
\begin{split}
a_{4t}=ab_{4t}.
\end{split}
 \end{align}

Similarly, by setting $s=t+1, k=N+1;\, s=t+2, k=N+3$ in (\ref{6.1}), we get the following equations:
\begin{align}
\begin{split}
    & a_{2t-N+1}b_0+a_{2t-N}b_2+\cdots+a_0b_{4t-2N+2}+\frac{2t+N+2}{2t+N+3}ab_{4t+2} \\
     &\ \ \ \ \ \ \ \ \ \ \ \ =a_0b_{2t-N+1}+a_2b_{2t-N}+\cdots+a_{4t-2N+2}b_0+\frac{2t+N+2}{2t+N+3}a_{4t+2};\label{6.6}
\end{split}
\end{align}
\begin{align}
\begin{split}
    & a_{2t-N+1}b_0+a_{2t-N}b_2+\cdots+a_0b_{4t-2N+2}+\frac{2t+N+4}{2t+N+5}ab_{4t+2} \\
     &\ \ \ \ \ \ \ \ \ \ \ \ =a_0b_{2t-N+1}+a_2b_{2t-N}+\cdots+a_{4t-2N+2}b_0+\frac{2t+N+4}{2t+N+5}a_{4t+2}.\label{6.7}
\end{split}
\end{align} 
Subtracting Equation (\ref{6.7})  from Equation (\ref{6.6}) yields that
$$\frac{2}{(2t+N+3)(2t+N+5)}ab_{4t+2}=\frac{2}{(2t+N+3)(2t+N+5)}a_{4t+2}.$$
Since $a\neq 0$ and $\frac{2}{(2t+N+3)(2t+N+5)}\neq0$ for any $t\geq 0$, we obtain that 
\begin{align}\label{6.8}
\begin{split}
a_{4t+2}=ab_{4t+2}.
\end{split}
 \end{align}
By (\ref{6.5}) and (\ref{6.8}), we have
\begin{align}\label{6.9}
\begin{split}
a_{2t}=ab_{2t}\ \ \mbox{for all}\  t\geq 0.
\end{split}
 \end{align}

Next, by setting $s=t,\,k=N-1;s=t+1,\,k=N+1; s=t+2,\,k=N+3$ in (\ref{6.2}), we get the following equations:
\begin{align}
\begin{split}
    & a_{2t-N}b_1+a_{2t-N-1}b_3+\cdots+a_0b_{4t-2N+1}+\frac{2t+N}{2t+N+1}ab_{4t+1}+\frac{N}{2N}a_{2t-N+1+\frac{N-1}{2}} \\
     &\ \ \ \ \ \ \ \ \ \ \ \ =a_1b_{2t-N}+a_3b_{2t-N-1}+\cdots+a_{4t-2N+1}b_0+\frac{N}{2N}ab_{2t-N+1+\frac{N-1}{2}}+\frac{2t+N}{2t+N+1}a_{4t+1};\label{6.10}
\end{split}
\end{align}
\begin{align}
\begin{split}
    & a_{2t-N}b_1+a_{2t-N-1}b_3+\cdots+a_0b_{4t-2N+1}+\frac{2t+N+2}{2t+N+3}ab_{4t+1}+\frac{N+4}{2N+4}a_{2t-N+1+\frac{N-1}{2}} \\
     &\ \ \ \ \ \ \ \ \ \ \ \ =a_1b_{2t-N}+a_3b_{2t-N-1}+\cdots+a_{4t-2N+1}b_0+\frac{N+4}{2N+4}ab_{2t-N+1+\frac{N-1}{2}}+\frac{2t+N+2}{2t+N+3}a_{4t+1};\label{6.11}
\end{split}
\end{align}
\begin{align}
\begin{split}
    & a_{2t-N}b_1+a_{2t-N-1}b_3+\cdots+a_0b_{4t-2N+1}+\frac{2t+N+4}{2t+N+5}ab_{4t+1}+\frac{N+8}{2N+8}a_{2t-N+1+\frac{N-1}{2}} \\
     &\ \ \ \ \ \ \ \ \ \ \ \ =a_1b_{2t-N}+a_3b_{2t-N-1}+\cdots+a_{4t-2N+1}b_0+\frac{N+8}{2N+8}ab_{2t-N+1+\frac{N-1}{2}}+\frac{2t+N+4}{2t+N+5}a_{4t+1}.\label{6.12}
\end{split}
\end{align}
Subtracting Equation (\ref{6.11})  from Equation (\ref{6.10}) and subtracting Equation (\ref{6.12})  from Equation (\ref{6.11}) yields that
 \begin{align*}
 \begin{cases}
\frac{2(ab_{4t+1}-a_{4t+1})}{(2t+N+1)(2t+N+3)}+\frac{4N}{2N(2N+4)}(a_{2t-N+1+\frac{N-1}{2}}-ab_{2t-N+1+\frac{N-1}{2}})=0,\\
\frac{2(ab_{4t+1}-a_{4t+1})}{(2t+N+3)(2t+N+5)}+\frac{4N}{(2N+4)(2N+8)}(a_{2t-N+1+\frac{N-1}{2}}-ab_{2t-N+1+\frac{N-1}{2}})=0.
 \end{cases}
 \end{align*}
Denoting $ab_{4t+3}-a_{4t+1}$ as $E$ and $a_{2t-N+1+\frac{N-1}{2}}-ab_{2t-N+1+\frac{N-1}{2}}$ as $F$, respectively, we obtain that
\begin{align}\label{6.14}
 \begin{cases}
\frac{2}{(2t+N+1)(2t+N+3)}E+\frac{4N}{2N(2N+4)}F=0,\\
\frac{2}{(2t+N+3)(2t+N+5)}E+\frac{4N}{(2N+4)(2N+8)}F=0.
 \end{cases}
 \end{align}
The coefficient matrix of (\ref{6.14}) is
\[\left(\begin{array}{cc}
         \frac{2}{(2t+N+1)(2t+N+3)} & \frac{4N}{2N(2N+4)} \\
         \frac{2}{(2t+N+3)(2t+N+5)} & \frac{4N}{(2N+4)(2N+8)} 
\end{array} \right)\]
and it is invertible. Thus $E=F=0$, we have
\begin{align}\label{6.15}
\begin{split}
a_{4t+1}=ab_{4t+1}\  \mbox{for all}\  t\geq 0.
\end{split}
 \end{align}

Similarly, by setting $s=t+1,\,k=N;s=t+2,\,k=N+2; s=t+3,\,k=N+4$ in (\ref{6.2}), we get the following equations:
\begin{align}
\begin{split}
    & a_{2t-N+1}b_1+a_{2t-N}b_3+\cdots+a_0b_{4t-2N+3}+\frac{2t+N+2}{2t+N+3}ab_{4t+3}+\frac{N+2}{2N+2}a_{2t+2-N+\frac{N-1}{2}} \\
     &\ \ \ \ \ \ \ \ \ \ \ \ =a_1b_{2t-N+1}+a_3b_{2t-N}+\cdots+a_{4t-2N+3}b_0+\frac{N+2}{2N+2}ab_{2t+2-N+\frac{N-1}{2}}+\frac{2t+N+2}{2t+N+3}a_{4t+3};\label{6.16}
\end{split}
\end{align}
\begin{align}
\begin{split}
    & a_{2t-N+1}b_1+a_{2t-N}b_3+\cdots+a_0b_{4t-2N+3}+\frac{2t+N+4}{2t+N+5}ab_{4t+3}+\frac{N+6}{2N+6}a_{2t+2-N+\frac{N-1}{2}} \\
     &\ \ \ \ \ \ \ \ \ \ \ \ =a_1b_{2t-N+1}+a_3b_{2t-N}+\cdots+a_{4t-2N+3}b_0+\frac{N+6}{2N+6}ab_{2t+2-N+\frac{N-1}{2}}+\frac{2t+N+4}{2t+N+5}a_{4t+3};\label{6.17}
\end{split}
\end{align}
\begin{align}
\begin{split}
    & a_{2t-N+1}b_1+a_{2t-N}b_3+\cdots+a_0b_{4t-2N+3}+\frac{2t+N+6}{2t+N+7}ab_{4t+3}+\frac{N+10}{2N+10}a_{2t+2-N+\frac{N-1}{2}} \\
     &\ \ \ \ \ \ \ \ \ \ \ \ =a_1b_{2t-N+1}+a_3b_{2t-N}+\cdots+a_{4t-2N+3}b_0+\frac{N+10}{2N+10}ab_{2t+2-N+\frac{N-1}{2}}+\frac{2t+N+6}{2t+N+7}a_{4t+3}.\label{6.18}
\end{split}
\end{align}
Subtracting Equation (\ref{6.17})  from Equation (\ref{6.16}) and subtracting Equation (\ref{6.18})  from Equation (\ref{6.17}) yields that
 \begin{align*}
 \begin{cases}
\frac{2(ab_{4t+3}-a_{4t+3})}{(2t+N+3)(2t+N+5)}+\frac{4N}{(2N+2)(2N+6)}(a_{2t+2-N+\frac{N-1}{2}}-ab_{2t+2-N+\frac{N-1}{2}})=0,\\
\frac{2(ab_{4t+3}-a_{4t+3})}{(2t+N+5)(2t+N+7)}+\frac{4N}{(2N+6)(2N+10)}(a_{2t+2-N+\frac{N-1}{2}}-ab_{2t+2-N+\frac{N-1}{2}})=0.
 \end{cases}
 \end{align*}
Denoting $ab_{4t+3}-a_{4t+3}$ as $X$ and $a_{2t+2-N+\frac{N-1}{2}}-ab_{2t+2-N+\frac{N-1}{2}}$ as $Y$, respectively, we obtain that
\begin{align}\label{6.19}
 \begin{cases}
\frac{2}{(2t+N+1)(2t+N+3)}X+\frac{4N}{(2N+2)(2N+6)}Y=0,\\
\frac{2}{(2t+N+3)(2t+N+5)}X+\frac{4N}{(2N+6)(2N+10)}Y=0.
 \end{cases}
 \end{align}
The coefficient matrix of (\ref{6.19}) is
\[\left(\begin{array}{cc}
         \frac{2}{(2t+N+1)(2t+N+3)} & \frac{4N}{(2N+2)(2N+6)} \\
         \frac{2}{(2t+N+3)(2t+N+5)} & \frac{4N}{(2N+6)(2N+10)} 
\end{array} \right)\]
and it is invertible. Thus $X=Y=0$, we have
\begin{align}\label{6.20}
\begin{split}
a_{4t+3}=ab_{4t+3}
\end{split}
 \end{align}
for all $t\geq 0$. By (\ref{6.15}) and (\ref{6.20}), we have
\begin{align}\label{6.21}
\begin{split}
a_{2t+1}=ab_{2t+1}
\end{split}
 \end{align}
for all $t\geq 0$. By (\ref{6.9}) and (\ref{6.21}) we obtain that $\varphi=a\psi$. Thus,  $f=ag$.\\

If $N$ is even, note that $B_fB_g=B_gB_f$ is equivalent to 
$$B_\varphi B_\psi z^{k}+aB_{{\overline{z}}^N}B_{\psi}z^{k}+ B_\varphi B_{{\overline{z}}^N} z^{k}= B_\psi B_\varphi z^{k}+aB_{\psi}B_{{\overline{z}}^N}z^{k}+B_{{\overline{z}}^N}B_{\varphi}z^{k}$$
for every  nonnegative integer $k$. Thus, we have
\begin{align*}
\begin{cases}
 B_\varphi B_\psi z^{k}=WT_\varphi W T_\psi z^{k}=WWT_{\varphi(z^2)}T_\psi z^{k}=WWT_{\varphi(z^2)\psi}z^{k},\\
 aB_{{\overline{z}}^N}B_\psi z^{k}=aW T_{{\overline{z}}^N}W T_\psi z^{k}, \\
 B_\varphi B_{{\overline{z}}^N}z^{k}=WT_\varphi W T_{{\overline{z}}^N}z^{k} =WWT_{\varphi(z^2)}T_{{\overline{z}}^N}z^{k}, \\
 B_\psi B_\varphi z^{k}=W T_\psi WT_\varphi z^{k}=WWT_{\psi(z^2)}T_\varphi z^{k}=WWT_{\psi(z^2)\varphi}z^{k},\\
 aB_\psi B_{{\overline{z}}^N} z^{k}=aWT_\psi W T_{{\overline{z}}^N}z^{k} =aWWT_{\psi(z^2)}T_{{\overline{z}}^N}z^{k}, \\
 B_{{\overline{z}}^N}B_\varphi z^{k}=WT_{{\overline{z}}^N} WT_{\varphi}z^k.
\end{cases}
\end{align*}
Since for every $k\geq0$, $B_\varphi B_\psi z^{k}, B_{{\overline{z}}^N}B_{\psi}z^{k}, B_\varphi B_{{\overline{z}}^N}z^{k}, B_\psi B_\varphi z^{k}, B_{\psi}B_{{\overline{z}}^N}z^{k}$ and $B_{{\overline{z}}^N}B_{\varphi}z^{k}$ are all analytic, letting 
$$B_\varphi B_\psi z^{2k}+aB_{{\overline{z}}^N}B_{\psi}z^{2k}+ B_\varphi B_{{\overline{z}}^N} z^{2k}:=\sum_{s=0}^{\infty}A_s z^s,$$ 
$$B_\psi B_\varphi z^{2k}+aB_{\psi}B_{{\overline{z}}^N}z^{2k}+B_{{\overline{z}}^N}B_{\varphi}z^{2k}:=\sum_{s=0}^{\infty}B_s z^s.$$
So we only need to consider the $4s$-th terms in $T_{\varphi(z^2)\psi}z^{2k},\ T_{\psi(z^2)\varphi}z^{2k},\ T_{\varphi(z^2)}T_{{\overline{z}}^N}z^{2k},\  T_{\psi(z^2)}T_{{\overline{z}}^N}z^{2k}$ and the $2s$-th terms in $T_{{\overline{z}}^N}W T_\varphi z^{2k},\ T_{{\overline{z}}^N}W T_\psi z^{2k}$.  Calculations show that
\begin{align*}
T_{\varphi(z^2)\psi}z^{2k}&=P\left(\varphi(z^2)\psi z^{2k}\right)\\
&=\sum_{n=0}^{\infty}\sum_{m=0}^{\infty}a_nb_mz^{2n+m+2k}\\
&:=\sum_{s=0}^{\infty}A_{1,s}z^s,  
\end{align*}
\begin{align*}
aT_{{\overline{z}}^N}W T_\psi z^{2k} &=aT_{{\overline{z}}^N}W\left(\sum_{m=0}^{\infty}b_m z^m z^{2k}\right)\\
&=aT_{{\overline{z}}^N}\left(\sum_{m=0}^{\infty}b_{2m}z^{m+k}\right)\\
&=aP\left({{\overline{z}}^N}\sum_{m=0}^{\infty}b_{2m}z^{m+k}\right)\\
      &=a\sum_{m=0}^{\infty}\frac{m+k+1-N}{m+k+1}b_{2m}z^{m+k-N}\\
      &:=\sum_{s=0}^{\infty}A_{2,s}z^s,  
\end{align*}
\begin{align*}
T_{\varphi(z^2)}T_{{\overline{z}}^N}z^{2k} &=T_{\varphi(z^2)}P({{\overline{z}}^N}z^{2k})\\
&=T_{\varphi(z^2)}\left(\frac{2k+1-N}{2k+1}z^{2k-N}\right)\\
      &=P\left(\sum_{n=0}^{\infty}\frac{2k+1-N}{2k+1}a_{n}z^{2n}z^{2k-N}\right)\\
      &=\sum_{n=0}^{\infty}\frac{2k+1-N}{2k+1}a_{n}z^{2n+2k-N}\\
      &:=\sum_{s=0}^{\infty}A_{3,s}z^s,  
\end{align*}
\begin{align*}
T_{\psi(z^2)\varphi}z^{2k}&=P\left(\psi(z^2)\varphi z^{2k}\right)\\
&=\sum_{n=0}^{\infty}\sum_{m=0}^{\infty}a_nb_mz^{n+2m+2k}\\
&:=\sum_{s=0}^{\infty}B_{1,s}z^s,  
\end{align*}
\begin{align*}
aT_{\psi(z^2)}T_{{\overline{z}}^N}z^{2k} &=aT_{\psi(z^2)}P({{\overline{z}}^N}z^{2k})\\
&=aT_{\psi(z^2)}\left(\frac{2k+1-N}{2k+1}z^{2k-N}\right)\\
      &=aP\left(\sum_{m=0}^{\infty}\frac{2k+1-N}{2k+1}b_{m}z^{2m}z^{2k-N}\right)\\
      &=a\sum_{m=0}^{\infty}\frac{2k+1-N}{2k+1}b_{m}z^{2m+2k-N}\\
      &:=\sum_{s=0}^{\infty}B_{2,s}z^s,  
\end{align*}
\begin{align*}
T_{{\overline{z}}^N}W T_\varphi z^{2k} &=T_{{\overline{z}}^N}W\left(\sum_{n=0}^{\infty}a_n z^n z^{2k}\right)\\
&=T_{{\overline{z}}^N}\left(\sum_{n=0}^{\infty}a_{2n}z^{n+k}\right)\\
&=P\left({{\overline{z}}^N}\sum_{n=0}^{\infty}a_{2n}z^{n+k}\right)\\
      &=\sum_{n=0}^{\infty}\frac{n+k+1-N}{n+k+1}a_{2n}z^{n+k-N}\\
      &:=\sum_{s=0}^{\infty}B_{3,s}z^s.  
\end{align*}
Using $B_\varphi B_\psi z^{k}+aB_{{\overline{z}}^N}B_{\psi}z^{k}+ B_\varphi B_{{\overline{z}}^N} z^{k}= B_\psi B_\varphi z^{k}+aB_{\psi}B_{{\overline{z}}^N}z^{k}+B_{{\overline{z}}^N}B_{\varphi}z^{k}$,  we obtain that $A_s=B_s$ for all $s\geq0$. 
Thus, 
$$A_s=A_{1,4s}+A_{2,2s}+A_{3,4s}=B_{1,4s}+B_{2,4s}+B_{3,2s}=B_s.$$

Now we are going to calculate the coefficients $A_s$ and $B_s$ defined above. Indeed, we have
\begin{align*}
  A_s&=A_{1,4s}+A_{2,2s}+A_{3,4s}\\
     &=a_{2s-k}b_0+a_{2s-k-1}b_2+\cdots+a_0b_{4s-2k}+\frac{2s+N}{2s+N+1}ab_{4s-2k+2N}+\frac{2k+1-N}{2k+1}a_{2s-k+\frac{N}{2}}\\
\end{align*}
and
\begin{align*}
  B_s&=B_{1,4s}+B_{2,4s}+B_{3,2s}\\
     &=a_{0}b_{2s-k}+a_{2}b_{2s-k-1}+\cdots+a_{4s-2k}b_{0}+\frac{2s+N}{2s+N+1}a_{4s-2k+2N}+\frac{2k+1-N}{2k+1}b_{2s-k+\frac{N}{2}}.
\end{align*}
This implies that
\begin{align}\label{6.22}
\begin{split}
  &a_{2s-k}b_0+a_{2s-k-1}b_2+\cdots+a_0b_{4s-2k}+\frac{2s+N}{2s+N+1}ab_{4s-2k+2N}+\frac{2k+1-N}{2k+1}a_{2s-k+\frac{N}{2}}\\
  &=a_{0}b_{2s-k}+a_{2}b_{2s-k-1}+\cdots+a_{4s-2k}b_{0}+\frac{2s+N}{2s+N+1}a_{4s-2k+2N}\\
  &\ \ \ \ \ \ \ \ \ \ \ \ \ \ \ \ +\frac{2k+1-N}{2k+1}b_{2s-k+\frac{N}{2}}.
\end{split}
 \end{align}

Similarly, let 
$$B_\varphi B_\psi z^{2k+1}+aB_{{\overline{z}}^N}B_{\psi}z^{2k+1}+ B_\varphi B_{{\overline{z}}^N} z^{2k+1}:=\sum_{s=0}^{\infty}C_s z^s,$$
$$ B_\psi B_\varphi z^{2k+1}+aB_{\psi}B_{{\overline{z}}^N}z^{2k+1}+B_{{\overline{z}}^N}B_{\varphi}z^{2k+1}:=\sum_{s=0}^{\infty}D_s z^s.$$
We also need to consider $4s$-th terms in $T_{\varphi(z^2)\psi}z^{2k+1},\ T_{\psi(z^2)\varphi}z^{2k+1},\ T_{\varphi(z^2)}T_{{\overline{z}}^N}z^{2k+1},\  T_{\psi(z^2)}T_{{\overline{z}}^N}z^{2k+1}$ and the $2s$-th terms in $T_{{\overline{z}}^N}W T_\varphi z^{2k+1},\ T_{{\overline{z}}^N}W T_\psi z^{2k+1}$.  Calculations show that
\begin{align*}
T_{\varphi(z^2)\psi}z^{2k+1}&=P\left(\varphi(z^2)\psi z^{2k+1}\right)\\
&=\sum_{n=0}^{\infty}\sum_{m=0}^{\infty}a_nb_mz^{2n+m+2k+1}\\
&:=\sum_{s=0}^{\infty}C_{1,s}z^s,  
\end{align*}
\begin{align*}
aT_{{\overline{z}}^N}W T_\psi z^{2k+1} &=aT_{{\overline{z}}^N}W\left(\sum_{m=0}^{\infty}b_m z^m z^{2k+1}\right)\\
&=aT_{{\overline{z}}^N}\left(\sum_{m=0}^{\infty}b_{2m+1}z^{m+k+1}\right)\\
      &=aP\left({{\overline{z}}^N}\sum_{m=0}^{\infty}b_{2m+1}z^{m+k+1}\right)\\
      &=a\sum_{m=0}^{\infty}\frac{m+k+2-N}{m+k+2}b_{2m+1}z^{m+k+1-N}\\
      &:=\sum_{s=0}^{\infty}C_{2,s}z^s,  
\end{align*}
\begin{align*}
T_{\varphi(z^2)}T_{{\overline{z}}^N}z^{2k+1} &=T_{\varphi(z^2)}P({{\overline{z}}^N}z^{2k+1})\\
&=T_{\varphi(z^2)}\left(\frac{2k+2-N}{2k+2}z^{2k+1-N}\right)\\
      &=P\left(\sum_{n=0}^{\infty}\frac{2k+2-N}{2k+2}a_{n}z^{2n}z^{2k+1-N}\right)\\
      &=\sum_{n=0}^{\infty}\frac{2k+2-N}{2k+2}a_{n}z^{2n+2k+1-N}\\
      &:=\sum_{s=0}^{\infty}C_{3,s}z^s,  
\end{align*}
\begin{align*}
T_{\psi(z^2)\varphi}z^{2k+1}&=P\left(\psi(z^2)\varphi z^{2k+1}\right)\\
&=\sum_{n=0}^{\infty}\sum_{m=0}^{\infty}a_nb_mz^{n+2m+2k+1}\\
&:=\sum_{s=0}^{\infty}D_{1,s}z^s,  
\end{align*}
\begin{align*}
aT_{\psi(z^2)}T_{{\overline{z}}^N}z^{2k+1} &=aT_{\psi(z^2)}P({{\overline{z}}^N}z^{2k+1})\\
&=aT_{\psi(z^2)}\left(\frac{2k+2-N}{2k+2}z^{2k+1-N}\right)\\
      &=aP\left(\sum_{m=0}^{\infty}\frac{2k+2-N}{2k+2}b_{m}z^{2m}z^{2k+1-N}\right)\\
      &=a\sum_{m=0}^{\infty}\frac{2k+2-N}{2k+2}b_{m}z^{2m+2k+1-N}\\
      &:=\sum_{s=0}^{\infty}D_{2,s}z^s,  
\end{align*}
\begin{align*}
T_{{\overline{z}}^N}W T_\varphi z^{2k+1} &=T_{{\overline{z}}^N}W\left(\sum_{n=0}^{\infty}a_n z^n z^{2k+1}\right)\\
&=T_{{\overline{z}}^N}\left(\sum_{n=0}^{\infty}a_{2n+1}z^{n+k+1}\right)\\
      &=P\left({{\overline{z}}^N}\sum_{n=0}^{\infty}a_{2n+1}z^{n+k+1}\right)\\
      &=\sum_{n=0}^{\infty}\frac{n+k+2-N}{n+k+2}a_{2n+1}z^{n+k+1-N}\\
      &:=\sum_{s=0}^{\infty}D_{3,s}z^s.  
\end{align*}
We also obtain that $C_s=D_s$ for all $s\geq 0$. Note that
$$WWT_{\varphi(z^2)}T_{{\overline{z}}^N}z^{2k+1}=WW\left(\sum_{n=0}^{\infty}\frac{2k+2-N}{2k+2}a_{n}z^{2n+2k+1-N}\right)=0,$$
$$WWT_{\psi(z^2)}T_{{\overline{z}}^N}z^{2k+1}=WW\left(\sum_{m=0}^{\infty}\frac{2k+2-N}{2k+2}b_{m}z^{2m+2k+1-N}\right)=0.$$\\
Thus, 
$$C_s=C_{1,4s}+C_{2,2s}=D_{1,4s}+D_{3,2s}=D_s.$$
We concluded that
\begin{align*}
  C_s&=C_{1,4s}+C_{2,2s}\\
     &=a_{2s-k-1}b_1+a_{2s-k-2}b_3+\cdots+a_0b_{4s-2k-1}+\frac{2s+N}{2s+N+1}ab_{4s-2k+2N-1} \\
\end{align*}
and
\begin{align*}
  D_s&=D_{1,4s}+D_{3,2s}\\
     &=a_{1}b_{2s-k-1}+a_{3}b_{2s-k-2}+\cdots+a_{4s-2k-1}b_{0}+\frac{2s+N}{2s+N+1}a_{4s-2k+2N-1}.
\end{align*}
This follows that
\begin{align}\label{6.23}
\begin{split}
&a_{2s-k-1}b_1+a_{2s-k-2}b_3+\cdots+a_0b_{4s-2k-1}+\frac{2s+N}{2s+N+1}ab_{4s-2k+2N-1}\\
     &=a_{1}b_{2s-k-1}+a_{3}b_{2s-k-2}+\cdots+a_{4s-2k-1}b_{0}+\frac{2s+N}{2s+N+1}a_{4s-2k+2N-1}.
\end{split}
 \end{align}

By setting $s=t+1, k=N+1;\, s=t+2, k=N+3$ in (\ref{6.22}), we get the following equations:
 \begin{align}
\begin{split}
  &a_{2t-N}b_1+a_{2t-N-1}b_3+\cdots+a_0b_{4t-2N+1}+\frac{2t+N+2}{2t+N+3}ab_{4t+1}\\
 &\ \ \  =a_1b_{2t-N}+a_3b_{2t-N-1}+\cdots+a_{4t-2N+1}b_0+\frac{2t+N+2}{2t+N+3}a_{4t+1};\label{6.24}
  \end{split}
   \end{align}  
\begin{align}
\begin{split}
  &a_{2t-N}b_1+a_{2t-N-1}b_3+\cdots+a_0b_{4t-2N+1}+\frac{2t+N+4}{2t+N+5}ab_{4t+1}\\
  &\ \ \ =a_1b_{2t-N}+a_3b_{2t-N-1}+\cdots+a_{4t-2N+1}b_0+\frac{2t+N+4}{2t+N+5}a_{4t+1}.\label{6.25}
 \end{split}
   \end{align}  
Subtracting Equation (\ref{6.25})  from Equation (\ref{6.24}) yields that
$$\frac{2}{(2t+N+3)(2t+N+5)}ab_{4t+1}=\frac{2}{(2t+N+3)(2t+N+5)}a_{4t+1}.$$
Since $a\neq 0$ and $\frac{2}{(2t+N+3)(2t+N+5)}\neq0$ for any $t\geq 0$, we obtain that 
\begin{align}\label{6.26}
\begin{split}
a_{4t+1}=ab_{4t+1}.
\end{split}
 \end{align}

Similarly, letting $s=t+1, k=N;\, s=t+2, k=N+2,$  respectively, in (\ref{6.23}), we get the following equations:
\begin{align}
\begin{split}
    & a_{2t-N+1}b_1+a_{2t-N}b_3+\cdots+a_0b_{4t-2N+3}+\frac{2t+N+2}{2t+N+3}ab_{4t+3} \\
     &\ \ \ \ \ \ \ \ \ \ \ \ =a_1b_{2t-N+1}+a_3b_{2t-N}+\cdots+a_{4t-2N+3}b_0+\frac{2t+N+2}{2t+N+3}a_{4t+3};\label{6.27}
\end{split}
\end{align}
\begin{align}
\begin{split}
    & a_{2t-N+1}b_1+a_{2t-N}b_3+\cdots+a_0b_{4t-2N+3}+\frac{2t+N+4}{2t+N+5}ab_{4t+3} \\
     &\ \ \ \ \ \ \ \ \ \ \ \ =a_1b_{2t-N+1}+a_3b_{2t-N}+\cdots+a_{4t-2N+3}b_0+\frac{2t+N+4}{2t+N+5}a_{4t+3}.\label{6.28}
\end{split}
\end{align}
Subtracting Equation (\ref{6.28})  from Equation (\ref{6.27}) yields that
$$\frac{2}{(2t+N+3)(2t+N+5)}ab_{4t+3}=\frac{2}{(2t+N+3)(2t+N+5)}a_{4t+3}.$$
Since $a\neq 0$ and $\frac{2}{(2t+N+3)(2t+N+5)}\neq0$ for any $t\geq 0$, we obtain that 
\begin{align}\label{6.29}
\begin{split}
a_{4t+3}=ab_{4t+3}.
\end{split}
 \end{align}
By (\ref{6.26}) and (\ref{6.29}), we have
\begin{align}\label{6.30}
\begin{split}
a_{2t+1}=ab_{2t+1}\  \mbox{for all}\  t\geq 0.
\end{split}
 \end{align}

Next, taking $s=t,\,k=N;s=t+1,\,k=N+2; s=t+2,\,k=N+4 $,  respectively, in (\ref{6.22}), we get the following equations:
\begin{align}
\begin{split}
    & a_{2t-N}b_0+a_{2t-N-1}b_2+\cdots+a_0b_{4t-2N}+\frac{2t+N}{2t+N+1}ab_{4t}+\frac{N+1}{2N+1}a_{2t-N+1+\frac{N}{2}} \\
     &\ \ \ \ \ \ \ \ \ \ \ \ =a_0b_{2t-N}+a_2b_{2t-N-1}+\cdots+a_{4t-2N}b_0+\frac{N+1}{2N+1}ab_{2t-N+1+\frac{N}{2}}+\frac{2t+N}{2t+N+1}a_{4t};\label{6.31}
\end{split}
\end{align}
\begin{align}
\begin{split}
    & a_{2t-N}b_0+a_{2t-N-1}b_2+\cdots+a_0b_{4t-2N}+\frac{2t+N+2}{2t+N+3}ab_{4t}+\frac{N+5}{2N+5}a_{2t-N+1+\frac{N}{2}} \\
     &\ \ \ \ \ \ \ \ \ \ \ \ =a_0b_{2t-N}+a_2b_{2t-N-1}+\cdots+a_{4t-2N}b_0+\frac{N+5}{2N+5}ab_{2t-N+1+\frac{N}{2}}+\frac{2t+N+2}{2t+N+3}a_{4t};\label{6.32}
\end{split}
\end{align}
\begin{align}
\begin{split}
    & a_{2t-N}b_0+a_{2t-N-1}b_2+\cdots+a_0b_{4t-2N}+\frac{2t+N+4}{2t+N+5}ab_{4t}+\frac{N+9}{2N+9}a_{2t-N+1+\frac{N}{2}} \\
     &\ \ \ \ \ \ \ \ \ \ \ \ =a_0b_{2t-N}+a_2b_{2t-N-1}+\cdots+a_{4t-2N}b_0+\frac{N+9}{2N+9}ab_{2t-N+1+\frac{N}{2}}+\frac{2t+N+4}{2t+N+5}a_{4t}.\label{6.33}
\end{split}
\end{align}
Subtracting Equation (\ref{6.32})  from Equation (\ref{6.31}) and subtracting Equation (\ref{6.33})  from Equation (\ref{6.32}) yields that
 \begin{align*}
 \begin{cases}
\frac{2(ab_{4t}-a_{4t})}{(2t+N+1)(2t+N+3)}+\frac{4N}{(2N+1)(2N+5)}(a_{2t-N+1+\frac{N}{2}}-ab_{2t-N+1+\frac{N}{2}})=0,\\
\frac{2(ab_{4t}-a_{4t})}{(2t+N+3)(2t+N+5)}+c(a_{2t-N+1+\frac{N}{2}}-ab_{2t-N+1+\frac{N}{2}})=0.
 \end{cases}
 \end{align*}
Denoting $ab_{4t}-a_{4t}$ as $E$ and $a_{2t-N+1+\frac{N}{2}}-ab_{2t-N+1+\frac{N}{2}}$ as $F$, respectively, we obtain that
\begin{align}\label{6.35}
 \begin{cases}
\frac{2}{(2t+N+1)(2t+N+3)}E+\frac{4N}{(2N+1)(2N+5)}F=0,\\
\frac{2}{(2t+N+3)(2t+N+5)}E+\frac{4N}{(2N+5)(2N+9)}F=0.
 \end{cases}
 \end{align}
The coefficient matrix of (\ref{6.35}) is
\[\left(\begin{array}{cc}
         \frac{2}{(2t+N+1)(2t+N+3)} & \frac{4N}{(2N+1)(2N+5)} \\
         \frac{2}{(2t+N+3)(2t+N+5)} & \frac{4N}{(2N+5)(2N+9)} 
\end{array} \right)\]
and it is invertible. Thus $E=F=0$, we have
\begin{align}\label{6.36}
\begin{split}
a_{4t}=ab_{4t}\  \mbox{for all}\  t\geq 0.
\end{split}
 \end{align}

Similarly, by setting $s=t+1,\,k=N+1;s=t+2,\,k=N+3; s=t+3,\,k=N+5$ in (\ref{6.22}), we get the following equations:
\begin{align}
\begin{split}
    & a_{2t-N+1}b_0+a_{2t-N}b_2+\cdots+a_0b_{4t-2N+2}+\frac{2t+N+2}{2t+N+3}ab_{4t+2}+\frac{N+3}{2N+3}a_{2t-N+1+\frac{N}{2}} \\
     &\ \ \ \ \ \ \ \ \ \ \ \ =a_0b_{2t-N+1}+a_2b_{2t-N}+\cdots+a_{4t-2N+2}b_0+\frac{N+3}{2N+3}ab_{2t-N+1+\frac{N}{2}}+\frac{2t+N+2}{2t+N+3}a_{4t+2};\label{6.37}
\end{split}
\end{align}
\begin{align}
\begin{split}
    & a_{2t-N+1}b_0+a_{2t-N}b_2+\cdots+a_0b_{4t-2N+2}+\frac{2t+N+4}{2t+N+5}ab_{4t+2}+\frac{N+7}{2N+7}a_{2t-N+1+\frac{N}{2}} \\
     &\ \ \ \ \ \ \ \ \ \ \ \ =a_0b_{2t-N+1}+a_2b_{2t-N}+\cdots+a_{4t-2N+2}b_0+\frac{N+7}{2N+7}ab_{2t-N+1+\frac{N}{2}}+\frac{2t+N+4}{2t+N+5}a_{4t+2};\label{6.38}
\end{split}
\end{align}
\begin{align}
\begin{split}
    & a_{2t-N+1}b_0+a_{2t-N}b_2+\cdots+a_0b_{4t-2N+2}+\frac{2t+N+6}{2t+N+7}ab_{4t+2}+\frac{N+11}{2N+11}a_{2t-N+1+\frac{N}{2}} \\
     &\ \ \ \ \ \ \ \ \ \ \ \ =a_0b_{2t-N+1}+a_2b_{2t-N}+\cdots+a_{4t-2N+2}b_0+\frac{N+11}{2N+11}ab_{2t-N+1+\frac{N}{2}}+\frac{2t+N+6}{2t+N+7}a_{4t+2}.\label{6.39}
\end{split}
\end{align}
Subtracting Equation (\ref{6.38})  from Equation (\ref{6.37}) and subtracting Equation (\ref{6.39})  from Equation (\ref{6.38}) yields that
 \begin{align*}
 \begin{cases}
\frac{2(ab_{4t+2}-a_{4t+2})}{(2t+N+3)(2t+N+5)}+\frac{4N}{(2N+3)(2N+7)}(a_{2t-N+1+\frac{N}{2}}-ab_{2t-N+1+\frac{N}{2}})=0,\\
\frac{2(ab_{4t+2}-a_{4t+2})}{(2t+N+5)(2t+N+7)}+\frac{4N}{(2N+7)(2N+11)}(a_{2t-N+1+\frac{N}{2}}-ab_{2t-N+1+\frac{N}{2}})=0.
 \end{cases}
 \end{align*}
Denoting $ab_{4t+2}-a_{4t+2}$ as $X$ and $a_{2t-N+1+\frac{N}{2}}-ab_{2t-N+1+\frac{N}{2}}$ as $Y$, respectively, we obtain that
\begin{align}\label{6.41}
 \begin{cases}
\frac{2}{(2t+N+3)(2t+N+5)}X+\frac{4N}{(2N+3)(2N+7)}Y=0,\\
\frac{2}{(2t+N+5)(2t+N+7)}X+\frac{4N}{(2N+7)(2N+11)}Y=0.
 \end{cases}
 \end{align}
The coefficient matrix of (\ref{6.41}) is
\[\left(\begin{array}{cc}
         \frac{2}{(2t+N+3)(2t+N+5)} & \frac{4N}{(2N+3)(2N+7)} \\
         \frac{2}{(2t+N+5)(2t+N+7)} & \frac{4N}{(2N+7)(2N+11)} 
\end{array} \right)\]
and it is invertible. Thus $X=Y=0$, we have
\begin{align}\label{6.42}
\begin{split}
a_{4t+2}=ab_{4t+2}
\end{split}
 \end{align}
for all $t\geq 0$. By (\ref{6.36}) and (\ref{6.42}), we have
\begin{align}\label{6.43}
\begin{split}
a_{2t}=ab_{2t}
\end{split}
 \end{align}
for all $t\geq 0$. By (\ref{6.30}) and (\ref{6.43}), we obtain that $\varphi=a\psi$. Thus,  $f=ag$.

If $a=0$, $b=1$, we have $f=\varphi$, $g={\overline{z}}^N+\psi$. This situation can be solved as described in \cite{[11]}.

\textbf{Step 2.}
If $b\neq1$ and $b\neq 0$, denoting $\frac{1}{b}\psi$ as $\phi$, let $h=\frac{1}{b}g={\overline{z}}^N+\phi$. Then $f$ and $h$ satisfying the case in step 1.
\end{proof}
\begin{rem}
  Let $\varphi$ and  $\psi$ be two bounded analytic functions on $\mathbb D$. Then  
  $$B_{\varphi+{\overline{z}}^2}B_{\psi+\overline{z}}\neq B_{\psi+\overline{z}}B_{\varphi+{\overline{z}}^2}$$
and  
    $$B_{\varphi+{\overline{z}}^3}B_{\psi+\overline{z}}\neq B_{\psi+\overline{z}}B_{\varphi+{\overline{z}}^3}.$$
\end{rem}
\begin{proof}
If $B_{\varphi+{\overline{z}}^2}B_{\psi+\overline{z}}\neq B_{\psi+\overline{z}}B_{\varphi+{\overline{z}}^2}$, take $\varphi=\psi=z$. We have
\begin{align}\label{7.1}
\begin{split}
B_{z+{{\overline{z}}^2}}B_{z+\overline{z}}=B_{z+\overline{z}}B_{z+{{\overline{z}}^2}}.
\end{split}
 \end{align}
(\ref{7.1}) is equivalent to
\begin{align*}
\begin{split}
B_{z}B_{\overline{z}}z^k+B_{{\overline{z}}^2}B_{z}z^k+B_{{\overline{z}}^2}B_{\overline{z}}z^k=B_{z}B_{{\overline{z}}^2}z^k+B_{\overline{z}}B_{z}z^k+B_{\overline{z}}B_{{\overline{z}}^2}z^k
\end{split}
 \end{align*}
for every  nonnegative integer $k$. Since
\begin{align*}
\begin{cases}
B_{z}B_{\overline{z}}z^k=WT_{z}WT_{\overline{z}}z^k =WWT_{z^2}T_{\overline{z}}z^k,\\
B_{{\overline{z}}^2}B_{z}z^k=WT_{{\overline{z}}^2}WT_{z}z^k,\\
B_{{\overline{z}}^2}B_{\overline{z}}z^k=WT_{{\overline{z}}^2}WT_{\overline{z}}z^k,\\
B_{z}B_{{\overline{z}}^2}z^k=WT_{z}WT_{{\overline{z}}^2}=WWT_{z^2}T_{{\overline{z}}^2}z^k,\\
B_{\overline{z}}B_{z}z^k=WT_{\overline{z}}WT_{z}z^k,\\
B_{\overline{z}}B_{{\overline{z}}^2}z^k=WT_{\overline{z}}WT_{{\overline{z}}^2}z^k.
\end{cases}
\end{align*}
Thus we only need to consider the $4s-th$ terms in $T_{z^2}T_{\overline{z}}z^k,\ T_{z^2}T_{{\overline{z}}^2}z^k$ and the $2s-th$ terms in $T_{{\overline{z}}^2}WT_{z}z^k$,\\$T_{{\overline{z}}^2}WT_{\overline{z}}z^k,$
$T_{\overline{z}}WT_{z}z^k,\,T_{\overline{z}}WT_{{\overline{z}}^2}z^k$. Calculations show that
\begin{itemize}
  \item[(1)] $\begin{aligned}T_{z^2}T_{\overline{z}}z^{2k} &=T_{z^2}\left(\frac{2k}{2k+1}z^{2k-1}\right)=\frac{2k}{2k+1}z^{2k+1};\end{aligned}$\\
  \item[(2)] $\begin{aligned}T_{{\overline{z}}^2}WT_{z}z^{2k} &=T_{{\overline{z}}^2}W(z^{2k+1})=0;\end{aligned}$\\
  \item[(3)]$\begin{aligned}T_{{\overline{z}}^2}WT_{\overline{z}}z^{2k} &= T_{{\overline{z}}^2}W\left(\frac{2k}{2k+1}z^{2k-1}\right)=0;\end{aligned}$\\
  \item[(4)] $\begin{aligned}T_{z^2}T_{{\overline{z}}^2}z^{2k} &=T_{z^2}\left(\frac{2k-1}{2k+1}z^{2k-2}\right)=\frac{2k-1}{2k+1}z^{2k};\end{aligned}$\\
  \item[(5)]$\begin{aligned}T_{\overline{z}}WT_{z}z^{2k} &=T_{\overline{z}}W(z^{2k+1})=0;\end{aligned}$\\
  \item[(6)] $\begin{aligned} T_{\overline{z}}WT_{{\overline{z}}^2}z^{2k} &=T_{\overline{z}}W\left(\frac{2k-1}{2k+1}z^{2k-2}\right)=T_{\overline{z}}\left(\frac{2k-1}{2k+1}z^{k-1}\right)=\frac{(k-1)(2k-1)}{k(2k+1)}z^{k-2}.\end{aligned}$\\
\end{itemize}
Note that $WWT_{z^2}T_{\overline{z}}z^{2k}=0$. Consider the constant terms of the following equation:
\begin{align}\label{7.2}
\begin{split}
B_{z}B_{\overline{z}}z^{2k}+B_{{\overline{z}}^2}B_{z}z^{2k}+B_{{\overline{z}}^2}B_{\overline{z}}z^{2k}=B_{z}B_{{\overline{z}}^2}z^{2k}+B_{\overline{z}}B_{z}z^{2k}+B_{\overline{z}}B_{{\overline{z}}^2}z^{2k}.
\end{split}
 \end{align}
We find that the left constant term of (\ref{7.2}) is equal to 0, but the right constant term of (\ref{7.2}) is equal to $\frac{3}{10}$, this is a contradiction! Thus, $B_{\varphi+{\overline{z}}^2}B_{\psi+\overline{z}}$ is not commutative.

Next, assuming $B_{\varphi+{\overline{z}}^3}B_{\psi+\overline{z}}\neq B_{\psi+\overline{z}}B_{\varphi+{\overline{z}}^3}$, and taking $\varphi=\psi=z$, we have
\begin{align}\label{7.3}
\begin{split}
B_{z+{{\overline{z}}^3}}B_{z+\overline{z}}=B_{z+\overline{z}}B_{z+{{\overline{z}}^3}}.
\end{split}
 \end{align}
(\ref{7.3}) is equivalent to
\begin{align*}
\begin{split}
B_{z}B_{\overline{z}}z^k+B_{{\overline{z}}^3}B_{z}z^k+B_{{\overline{z}}^3}B_{\overline{z}}z^k=B_{z}B_{{\overline{z}}^3}z^k+B_{\overline{z}}B_{z}z^k+B_{\overline{z}}B_{{\overline{z}}^3}z^k
\end{split}
 \end{align*}
for every  nonnegative integer $k$. Since
\begin{align*}
\begin{cases}
B_{z}B_{\overline{z}}z^k=WT_{z}WT_{\overline{z}}z^k =WWT_{z^2}T_{\overline{z}}z^k,\\
B_{{\overline{z}}^3}B_{z}z^k=WT_{{\overline{z}}^3}WT_{z}z^k,\\
B_{{\overline{z}}^3}B_{\overline{z}}z^k=WT_{{\overline{z}}^3}WT_{\overline{z}}z^k,\\
B_{z}B_{{\overline{z}}^3}z^k=WT_{z}WT_{{\overline{z}}^3}=WWT_{z^2}T_{{\overline{z}}^3}z^k,\\
B_{\overline{z}}B_{z}z^k=WT_{\overline{z}}WT_{z}z^k,\\
B_{\overline{z}}B_{{\overline{z}}^3}z^k=WT_{\overline{z}}WT_{{\overline{z}}^3}z^k.
\end{cases}
\end{align*}
Thus we only need to consider the $4s-th$ terms in $T_{z^2}T_{\overline{z}}z^k,\ T_{z^2}T_{{\overline{z}}^2}z^k$ and the $2s-th$ terms in $T_{{\overline{z}}^3}WT_{z}z^k$\\$T_{{\overline{z}}^3}WT_{\overline{z}}z^k,$
$T_{\overline{z}}WT_{z}z^k,\,T_{\overline{z}}WT_{{\overline{z}}^3}z^k$. Calculations show that 
\begin{itemize}
  \item[(1')] $\begin{aligned}T_{z^2}T_{\overline{z}}z^{2k+1} &=T_{z^2}\left(\frac{2k+1}{2k+2}z^{2k}\right)=\frac{2k+1}{2k+2}z^{2k+2};\end{aligned}$\\
  \item[(2')] $\begin{aligned}T_{{\overline{z}}^3}WT_{z}z^{2k+1} &=T_{{\overline{z}}^3}W(z^{2k+2})=T_{{\overline{z}}^3}(z^{k+1})=\frac{k-1}{k+2}z^{k-2};\end{aligned}$\\
  \item[(3')]$\begin{aligned}T_{{\overline{z}}^3}WT_{\overline{z}}z^{2k+1} &= T_{{\overline{z}}^3}W\left(\frac{2k+1}{2k+2}z^{2k}\right)=T_{{\overline{z}}^3}\left(\frac{2k+1}{2k+2}z^k\right)=\frac{(k-2)(2k+1)}{(k+1)(2k+2)}z^{k-3};\end{aligned}$\\
  \item[(4')] $\begin{aligned}T_{z^2}T_{{\overline{z}}^3}z^{2k+1} &=T_{z^2}\left(\frac{2k-1}{2k+2}z^{2k-2}\right)=\frac{2k-1}{2k+2}z^{2k};\end{aligned}$\\
  \item[(5')]$\begin{aligned}T_{\overline{z}}WT_{z}z^{2k+1} &=T_{\overline{z}}(z^{k+1})=\frac{k+1}{k+2}z^k;\end{aligned}$\\
  \item[(6')] $\begin{aligned} T_{\overline{z}}WT_{{\overline{z}}^3}z^{2k+1} &=T_{\overline{z}}W\left(\frac{2k-1}{2k+2}z^{2k-2}\right)=T_{\overline{z}}\left(\frac{2k-1}{2k+2}z^{k-1}\right)=\frac{(k-1)(2k-1)}{k(2k+2)} z^{k-2}.\end{aligned}$\\
\end{itemize}
Consider the constant terms of the following equation:
\begin{align}\label{7.4}
\begin{split}
B_{z}B_{\overline{z}}z^{2k+1}+B_{{\overline{z}}^3}B_{z}z^{2k+1}+B_{{\overline{z}}^3}B_{\overline{z}}z^{2k+1}=B_{z}B_{{\overline{z}}^3}z^{2k+1}+B_{\overline{z}}B_{z}z^{2k+1}+B_{\overline{z}}B_{{\overline{z}}^3}z^{2k+1}.
\end{split}
 \end{align}
We find that the left constant term of (\ref{7.4}) is equal to 5, but the right constant term of (\ref{7.4}) is equal to 9, this is a contradiction! Thus, $B_{\varphi+{\overline{z}}^3}B_{\psi+\overline{z}}$ is not commutative.
\end{proof}
Based on Lemma \ref{a} and Lemma \ref{b}, we proceed to replace the symbols of two slant Toeplitz operators with $\overline{p}+\varphi$ and $\overline{p}+\psi$, respectively, which leads to the following theorem. 
\begin{thm}\label{c}
Let $f=\overline{p}+\varphi,\,g=\overline{p}+\psi$, $\varphi$ and  $\psi$ be two bounded analytic functions on $\mathbb D$, and p be an analytic polynomian. Then $B_fB_g=B_gB_f$ if and only if $\varphi=\psi$.
\end{thm}
\begin{proof}
Let us show the sufficiency first. If $f=g$, the result is trivial. To prove the necessity, we suppose that
$$\varphi(z)=\sum_{n=0}^{\infty}a_n z^n, \ \ \ \ \ \psi(z)=\sum_{m=0}^{\infty}b_m z^m, \ \ \ \ \ \ \overline{p}(z)=\sum_{j=0}^{N} c_{j} {\overline{z}}^j, $$
where $N\geq1$ is an integer. Note that $B_ {\overline{p}+\varphi}B_{\overline{p}+\psi}=B_{\overline{p}+\psi}B_{\overline{p}+\varphi}$ is equivalent to
$$B_{\varphi}B_{\psi}z^k+B_{\overline{p}}B_{\psi-\varphi}z^k=B_{\psi}B_{\varphi}z^k+B_{\psi-\varphi}B_{\overline{p}}z^k,\,k\in{\Z}^+.$$
For convenience, define
$$h(z):=\psi(z)-\varphi(z)=\sum_{i=0}^{\infty}(b_i-a_i) z^i:=\sum_{i=0}^{\infty}d_i z^i.$$
Then $B_{\overline{p}+\varphi}B_{\overline{p}+\psi}=B_{\overline{p}+\psi}B_{\overline{p}+\varphi}$ is equivalent to
$$B_{\varphi}B_{\psi}z^k+B_{\overline{p}}B_{h}z^k=B_{\psi}B_{\varphi}z^k+B_{h}B_{\overline{p}}z^k,\,k\in{\Z}^+.$$
Then we have
\begin{align*}
\begin{cases}
 B_\varphi B_\psi z^{k}=WT_\varphi WT_\psi z^{k}=WWT_{\varphi(z^2)}T_\psi z^{k}=WWT_{\varphi(z^2)\psi}z^{k},\\
 B_{\overline{p}}B_h z^{k}=WT_{\overline{p}}WT_h z^{k}, \\
 B_\psi B_\varphi z^{k}=WT_\psi WT_\varphi z^{k}=WWT_{\psi(z^2)}T_\varphi z^{k}=WWT_{\psi(z^2)\varphi}z^{k},\\
 B_h B_{\overline{p}} z^{k}=WT_h WT_{\overline{p}}z^{k} =WWT_{h(z^2)}T_{\overline{p}}z^{k}.
\end{cases}
\end{align*}
Since for every $k\geq0$, $B_\varphi B_\psi z^{k},B_{\overline{p}}B_h z^{k},B_\psi B_\varphi z^{k}$ and $B_h B_{\overline{p}} z^{k}$ are all analytic, letting
$$B_\varphi B_\psi z^{2k}+B_{\overline{p}}B_h z^{2k}:=\sum_{s=0}^{\infty}A_s z^s, \ \ \ \ B_\psi B_\varphi z^{2k}+B_h B_{\overline{p}}z^{2k}:=\sum_{s=0}^{\infty}B_s z^s.$$
The definition of the operator $W$ implies that
$$WWz^{4s}=Wz^{2s}=z^s,\,s\geq0.$$
So we only need to consider the $4s-th$ terms in  $T_{\varphi(z^2)\psi}z^{2k},\,T_{\psi(z^2)\varphi}z^{2k},\,T_{h(z^2)}T_{\overline{p}}z^{2k}$ and the $2s-th$ terms in $T_ {\overline{p}}WT_hz^{2k}$. Calculations show that  
\begin{align*}
T_{\varphi(z^2)\psi}z^{2k}&=P\left(\varphi(z^2)\psi z^{2k}\right)\\
&=\sum_{n=0}^{\infty}\sum_{m=0}^{\infty}a_nb_mz^{2n+m+2k}\\
&:=\sum_{s=0}^{\infty}A_{1,s}z^s,  
\end{align*}
\begin{align*}
T_{\overline{p}}WT_h z^{2k} &=T_{\overline{p}}W\left(\sum_{i=0}^{\infty}d_i z^i z^{2k}\right)\\
&=T_{\overline{p}}\left(\sum_{i=0}^{\infty}d_{2i}z^{i+k}\right)\\
      &=P\left(\sum_{j=0}^{N} c_{j} {\overline{z}}^j\sum_{i=0}^{\infty}d_{2i}z^{i+k}\right)\\
      &:=\sum_{s=0}^{\infty}A_{2,s}z^s,  
\end{align*}
\begin{align*}
T_{\psi(z^2)\varphi}z^{2k}&=P\left(\psi(z^2)\varphi z^{2k}\right)\\
&=\sum_{n=0}^{\infty}\sum_{m=0}^{\infty}a_nb_mz^{n+2m+2k}\\
&:=\sum_{s=0}^{\infty}B_{1,s}z^s,  
\end{align*}
\begin{align*}
T_{h(z^2)}T_{\overline{p}}z^{2k} &=T_{h(z^2)}P(\sum_{j=0}^{N} c_{j} {\overline{z}}^jz^{2k})\\
&=\sum_{i=0}^{\infty}d_iz^{2i}\sum_{j=0}^{K_{2k}}c_j \frac{2k+1-j}{2k+1}z^{2k-j}\\
&:=\sum_{s=0}^{\infty}B_{2,s}z^s,  
\end{align*}
where $K_{2k}=\min\{2k,N\}$. Using $B_ {\varphi}B_{\psi}z^k+B_{\overline{p}}B_{\psi-\varphi}z^k=B_{\psi}B_{\varphi}z^k+B_{\psi-\varphi}B_{\overline{p}}z^k$, we obtain that$A_s=B_s$ for all $s\geq0$. Therefore,
$$A_s=A_{1,4s}+A_{2,2s}=B_{1,4s}+B_{2,4s}=B_s.$$

Now we are going to calculate the coefficients $A_s$ and $B_s$ defined above. Indeed, we have
\begin{align*}
A_s &=A_{1,4s}+A_{2,2s}\\&=   
    \begin{cases}
    0, & 2s-k<-N,\vspace{2mm}\\
    c_Nd_0, & 2s-k=-N,\vspace{2mm}\\
     \ \ \ \ \ \ \ \ \cdots\vspace{2mm}\\
    a_0b_0+(c_0d_0+\frac{2s+1}{2s+2}c_1d_2+\frac{2s+1}{2s+3}c_2d_4+\cdots+\frac{2s+1}{2s+N+1}c_Nd_{2N}),& 2s-k=0,\vspace{2mm}\\
    a_1b_0+a_0b_2+(c_0d_2+\frac{2s+1}{2s+2}c_1d_4+\frac{2s+1}{2s+3}c_2d_6+\cdots+\frac{2s+1}{2s+N+1}c_Nd_{2N+2}),& 2s-k=1,\vspace{2mm}\\
    \ \ \ \ \ \ \ \ \cdots\vspace{2mm}\\
    a_{2s-k}b_0+a_{2s-k-1}b_2+\cdots+a_0b_{4s-2k}+(c_0d_{4s-2k}+\frac{2s+1}{2s+2}c_1d_{4s-2k+2}+\vspace{2mm}\\
 \frac{2s+1}{2s+3}c_2d_{4s-2k+4}+\cdots+\frac{2s+1}{2s+N+1}c_Nd_{4s-2k+2N})
    \end{cases}
\end{align*}
and
\begin{align*}
B_s &=B_{1,4s}+B_{2,4s}\\&=   
    \begin{cases}
    0, & 2s-k<-\frac{M_{2k}}{2},\vspace{2mm}\\
    c_{M_{2k}}d_0, & 2s-k=-\frac{M_{2k}}{2},\vspace{2mm}\\
     \ \ \ \ \ \ \ \ \cdots\vspace{2mm}\\
    a_0b_0+(c_0d_0+\frac{2k-1}{2k+1}c_2d_1+\frac{2k-3}{2k+1}c_4d_2+\cdots+\frac{2k-M_{2k}+1}{2k+1}c_{M_{2k}}d_{\frac{M_{2k}}{2}}),& 2s-k=0,\vspace{2mm}\\
    a_1b_0+a_0b_2+(c_0d_1+\frac{2k-1}{2k+1}c_2d_2+\frac{2k-3}{2k+1}c_4d_3+\cdots+\frac{2k-M_{2k}+1}{2k+1}c_{M_{2k}}d_{\frac{M_{2k}+2}{2}}),& 2s-k=1,\vspace{2mm}\\
    \ \ \ \ \ \ \ \ \cdots\vspace{2mm}\\
    a_0 b_{2s-k}+a_2b_{2s-k-1}+\cdots+a_{4s-2k}b_0+(c_0d_{2s-k}+\frac{2k-1}{2k+1}c_2d_{2s-k+1}+\vspace{2mm}\\
 \frac{2k-3}{2k+1}c_4d_{2s-k+2}+\cdots+\frac{2k-M_{2k}+1}{2k+1}c_{M_{2k}}d_{2s-k+\frac{M_{2k}}{2}}),
    \end{cases}
\end{align*}
where
\begin{align*}
M_{2k}=   
    \begin{cases}
    2k, & \mbox{if}\ 2k\leq N,\vspace{2mm}\\
    N-1, & \mbox{if}\ 2k>N, N\ \mbox{is odd},\vspace{2mm}\\
    N,   &\mbox{if}\ 2k>N, N\ \mbox{is even}.
    \end{cases}
\end{align*}
This implies that
\begin{equation}\label{9.1}
 \begin{cases}
 c_Nd_0=c_{M_{2k}}d_0,& 2s-k=-N,\vspace{2mm}\\
 \ \ \ \ \  \ \ \ \ \ \ \ \ \ \ \cdots\vspace{2mm}\\
 a_0b_0+(c_0d_0+\frac{2s+1}{2s+2}c_1d_2+\frac{2s+1}{2s+3}c_2d_4+\cdots+\frac{2s+1}{2s+N+1}c_Nd_{2N})\vspace{2mm}\\
=a_0b_0+(c_0d_0+\frac{2k-1}{2k+1}c_2d_1+\frac{2k-3}{2k+1}c_4d_2+\cdots+\frac{2k-M_{2k}+1}{2k+1}c_{M_{2k}}d_{\frac{M_{2k}}{2}}),&2s-k=0,\vspace{2mm}\\
a_1b_0+a_0b_2+(c_0d_2+\frac{2s+1}{2s+2}c_1d_4+\frac{2s+1}{2s+3}c_2d_6+\cdots+\frac{2s+1}{2s+N+1}c_Nd_{2N+2})\vspace{2mm}\\
=a_1b_0+a_0b_2+(c_0d_1+\frac{2k-1}{2k+1}c_2d_2+\frac{2k-3}{2k+1}c_4d_3+\cdots\vspace{2mm}\\
+\frac{2k-M_{2k}+1}{2k+1}c_{M_{2k}}d_{\frac{M_{2k}+2}{2}}),&  2s-k=1,\vspace{2mm}\\
a_{2s-k}b_0+a_{2s-k-1}b_2+\cdots+a_0b_{4s-2k}+(c_0d_{4s-2k}+\frac{2s+1}{2s+2}c_1d_{4s-2k+2}\vspace{2mm}\\ +\frac{2s+1}{2s+3}c_2d_{4s-2k+4}+\cdots+\frac{2s+1}{2s+N+1}c_Nd_{4s-2k+2N})\vspace{2mm}\\
=a_0 b_{2s-k}+a_2b_{2s-k-1}+\cdots+a_{4s-2k}b_0+(c_0d_{2s-k}+\frac{2k-1}{2k+1}c_2d_{2s-k+1}\vspace{2mm}\\ +\frac{2k-3}{2k+1}c_4d_{2s-k+2}+\cdots+\frac{2k-M_{2k}+1}{2k+1}c_{M_{2k}}d_{2s-k+\frac{M_{2k}}{2}}).
\end{cases}
\end{equation}

Similarly, we define
$$B_\varphi B_\psi z^{2k+1}+B_{\overline{p}}B_h z^{2k+1}:=\sum_{s=0}^{\infty}C_s z^s, \ \ \ \ B_\psi B_\varphi z^{2k+1}+B_h B_{\overline{p}}z^{2k+1}:=\sum_{s=0}^{\infty}D_s z^s.$$
We also need to consider the $4s-th$ terms in $T_ {\varphi(z^2)\psi}z^{2k+1},\,T_{\psi(z^2)\varphi}z^{2k+1},\,T_{h(z^2)}T_{\overline{p}}z^{2k+1}$ and the $2s-th$ terms in $T_ {\overline{p}}WT_hz^{2k+1}$. Calculations show that  
\begin{align*}
T_{\varphi(z^2)\psi}z^{2k+1}=P\left(\varphi(z^2)\psi z^{2k+1}\right)=\sum_{n=0}^{\infty}\sum_{m=0}^{\infty}a_nb_mz^{2n+m+2k+1}:=\sum_{s=0}^{\infty}C_{1,s}z^s,  
\end{align*}
\begin{align*}
T_{\overline{p}}WT_h z^{2k+1} &=T_{\overline{p}}W\left(\sum_{i=0}^{\infty}d_i z^i z^{2k+1}\right)\\
&=T_{\overline{p}}\left(\sum_{i=0}^{\infty}d_{2i+1}z^{i+k+1}\right)\\
      &=P\left(\sum_{j=0}^{N} c_{j} {\overline{z}}^j\sum_{i=0}^{\infty}d_{2i+1}z^{i+k+1}\right):=\sum_{s=0}^{\infty}C_{2,s}z^s,  
\end{align*}
\begin{align*}
T_{\psi(z^2)\varphi}z^{2k}=P\left(\psi(z^2)\varphi z^{2k+1}\right)=\sum_{n=0}^{\infty}\sum_{m=0}^{\infty}a_nb_mz^{n+2m+2k+1}:=\sum_{s=0}^{\infty}D_{1,s}z^s,  
\end{align*}
\begin{align*}
T_{h(z^2)}T_{\overline{p}}z^{2k}& =T_{h(z^2)}P(\sum_{j=0}^{N} c_{j} {\overline{z}}^jz^{2k+1})\\
&=\sum_{i=0}^{\infty}d_iz^{2i}\sum_{j=0}^{K_{2k+1}}c_j \frac{2k+2-j}{2k+2}z^{2k+1-j}\\
&:=\sum_{s=0}^{\infty}D_{2,s}z^s, 
\end{align*}
where $K_{2k+1}=\min\{2k+1,N\}$. Using $B_ {\varphi}B_{\psi}z^k+B_{\overline{p}}B_{\psi-\varphi}z^k=B_{\psi}B_{\varphi}z^k+B_{\psi-\varphi}B_{\overline{p}}z^k$, we have $C_s=D_s$, for all $s\geq0$. Therefore,
$$C_s=C_{1,4s}+C_{2,2s}=D_{1,4s}+D_{2,4s}=D_s.$$
We conclude that
\begin{align*}
C_s &=C_{1,4s}+C_{2,2s}\\&=   
    \begin{cases}
    0, & 2s-k-1<-N,\vspace{2mm}\\
    c_Nd_1, & 2s-k-1=-N,\vspace{2mm}\\
     \ \ \ \ \ \ \ \ \cdots\vspace{2mm}\\
    a_0b_1+(c_0d_1+\frac{2s+1}{2s+2}c_1d_3+\frac{2s+1}{2s+3}c_2d_5+\cdots+\frac{2s+1}{2s+N+1}c_Nd_{2N+1}),& 2s-k-1=0,\vspace{2mm}\\
    a_1b_1+a_0b_3+(c_0d_3+\frac{2s+1}{2s+2}c_1d_5+\frac{2s+1}{2s+3}c_2d_7+\cdots+\frac{2s+1}{2s+N+1}c_Nd_{2N+3}),& 2s-k-1=1,\vspace{2mm}\\
    \ \ \ \ \ \ \ \ \cdots\vspace{2mm}\\
    a_{2s-k-1}b_1+a_{2s-k-2}b_3+\cdots+a_0b_{4s-2k-1}+(c_0d_{4s-2k-1}+\frac{2s+1}{2s+2}c_1d_{4s-2k+1}+\vspace{2mm}\\
 \frac{2s+1}{2s+3}c_2d_{4s-2k+3}+\cdots+\frac{2s+1}{2s+N+1}c_Nd_{4s-2k+2N-1})
    \end{cases}
\end{align*}
and
\begin{align*}
D_s &=D_{1,4s}+D_{2,4s}\\&=   
    \begin{cases}
    0, & 2s-k-1<-\frac{L_{2k+1}+1}{2},\vspace{2mm}\\
    c_{L_{2k+1}}d_0, & 2s-k-1=-\frac{L_{2k+1}+1}{2},\vspace{2mm}\\
     \ \ \ \ \ \ \ \ \cdots\vspace{2mm}\\
    a_1b_0+(c_1d_1+\frac{2k-1}{2k+2}c_3d_2+\frac{2k-3}{2k+2}c_5d_3+\cdots\vspace{2mm}\\
    +\frac{2k-L_{2k+1}+2}{2k+2}c_{L_{2k+1}}d_{\frac{L_{2k+1}+1}{2}}),& 2s-k-1=0,\vspace{2mm}\\
    a_1b_0+a_3b_0+(c_1d_2+\frac{2k-1}{2k+2}c_3d_3+\frac{2k-3}{2k+2}c_5d_4+\cdots\vspace{2mm}\\
    +\frac{2k-L_{2k+1}+2}{2k+2}c_{L_{2k+1}}d_{\frac{L_{2k+1}+3}{2}}),& 2s-k-1=1,\vspace{2mm}\\
    \ \ \ \ \ \ \ \ \cdots\vspace{2mm}\\
    a_1 b_{2s-k-1}+a_3b_{2s-k-2}+\cdots+a_{4s-2k-1}b_0+(c_1d_{2s-k}\vspace{2mm}\\
+\frac{2k-1}{2k+2}c_3d_{2s-k+1}+\frac{2k-3}{2k+2}c_5d_{2s-k+2}+\cdots\vspace{2mm}\\
+\frac{2k-L_{2k+1}+2}{2k+2}c_{L_{2k+1}}d_{2s-\frac{2k+1-L_{2k+1}}{2}}),
    \end{cases}
\end{align*}
where
\begin{align*}
L_{2k+1}=   
    \begin{cases}
    2k+1, & \mbox{if}\ 2k+1\leq N,\vspace{2mm}\\
    N-1, & \mbox{if}\ 2k+1>N, N \ \mbox{is even},\vspace{2mm}\\
    N,&\mbox{if}\ 2k+1>N, N\ \mbox{is odd}.
    \end{cases}
\end{align*}
This follows that
\begin{equation}\label{9.2}
 \begin{cases}
 \ \ \ \ \  \ \ \ \ \ \ \ \ \ \ \cdots\vspace{2mm}\\
a_0b_1+(c_0d_1+\frac{2s+1}{2s+2}c_1d_3+\frac{2s+1}{2s+3}c_2d_5+\cdots+\frac{2s+1}{2s+N+1}c_Nd_{2N+1})\vspace{2mm}\\
=a_1b_0+(c_1d_1+\frac{2k-1}{2k+2}c_3d_2+\frac{2k-3}{2k+2}c_5d_3+\cdots+\frac{2k-L_{2k+1}+2}{2k+2}c_{L_{2k+1}}d_{\frac{L_{2k+1}+1}{2}}),& 2s-k-1=0,\vspace{2mm}\\
a_1b_1+a_0b_3+(c_0d_3+\frac{2s+1}{2s+2}c_1d_5+\frac{2s+1}{2s+3}c_2d_7+\cdots+\frac{2s+1}{2s+N+1}c_Nd_{2N+3})\vspace{2mm}\\
=a_1b_0+a_3b_0+(c_1d_2+\frac{2k-1}{2k+2}c_3d_3+\frac{2k-3}{2k+2}c_5d_4+\cdots+\frac{2k-L_{2k+1}+2}{2k+2}c_{L_{2k+1}}d_{\frac{L_{2k+1}+3}{2}}),& 2s-k-1=1,\vspace{2mm}\\
 \ \ \ \ \  \ \ \ \ \ \ \ \ \ \ \cdots\vspace{2mm}\\
a_{2s-k-1}b_1+a_{2s-k-2}b_3+\cdots+a_0b_{4s-2k-1}+(c_0d_{4s-2k-1}+\frac{2s+1}{2s+2}c_1d_{4s-2k+1}+\vspace{2mm}\\
\frac{2s+1}{2s+3}c_2d_{4s-2k+3}+\cdots+\frac{2s+1}{2s+N+1}c_Nd_{4s-2k+2N-1})\vspace{2mm}\\
=a_1 b_{2s-k-1}+a_3b_{2s-k-2}+\cdots+a_{4s-2k-1}b_0+(c_1d_{2s-k}+\frac{2k-1}{2k+2}c_3d_{2s-k+1}+\vspace{2mm}\\ \frac{2k-3}{2k+2}c_5d_{2s-k+2}+\cdots+\frac{2k-L_{2k+1}+2}{2k+2}c_{L_{2k+1}}d_{2s-\frac{2k+1-L_{2k+1}}{2}}).
\end{cases}
\end{equation}

Let us proceed in two steps.

\textbf{Step 1.} We first divide the proof into two substeps.

{\textbf{Substep 1.1}}
Suppose $2t>N$. By setting $s=t,k=2t;s=t+1,k=2t+2;s=t+2,k=2t+4$ in (\ref{9.1}) and (\ref{9.2}), we obtain the following equations:
\begin{align}\label{9.3}
 \begin{cases}
 &a_0b_0+(c_0d_0+\frac{2t+1}{2t+2}c_1d_2+\cdots+\frac{2t+1}{2t+N+1}c_Nd_{2N})\vspace{2mm}\\
 &=a_0b_0+(c_0d_0+\frac{4t-1}{4t+1}c_2b_1+\cdots+\frac{4t+1-M_{2k}}{4t+1}c_{M_{2k}}d_{\frac{M_{2k}}{2}}),\vspace{2mm}\\
 &\frac{2t+1}{2t+2}c_1d_1+\frac{2t+1}{2t+3}c_2d_3+\cdots+\frac{2t+1}{2t+N+1}c_Nd_{2N-1}\vspace{2mm}\\
 &=\frac{4t+1}{4t+2}c_1d_0+\frac{4t-1}{4t+2}c_3d_1+\cdots+\frac{4t+2-L_{2k+1}}{4t+2}c_{L_{2k+1}}d_{\frac{L_{2k+1}-1}{2}};
 \end{cases}
\end{align}
\begin{align}\label{9.4}
 \begin{cases}
 &a_0b_0+(c_0d_0+\frac{2t+3}{2t+4}c_1d_2+\cdots+\frac{2t+3}{2t+N+3}c_Nd_{2N})\\
 &=a_0b_0+(c_0d_0+\frac{4t+3}{4t+5}c_2b_1+\cdots+\frac{4t+5-M_{2k}}{4t+5}c_{M_{2k}}d_{\frac{M_{2k}}{2}}),\vspace{1mm}\\
 &\frac{2t+3}{2t+4}c_1d_1+\frac{2t+3}{2t+5}c_2d_3+\cdots+\frac{2t+3}{2t+N+3}c_Nd_{2N-1}\\
 &=\frac{4t+5}{4t+6}c_1d_0+\frac{4t+3}{4t+6}c_3d_1+\cdots+\frac{4t+6-L_{2k+1}}{4t+6}c_{L_{2k+1}}d_{\frac{L_{2k+1}-1}{2}};
 \end{cases}
\end{align}
\begin{align}\label{9.5}
 \begin{cases}
 &a_0b_0+(c_0d_0+\frac{2t+5}{2t+6}c_1d_2+\cdots+\frac{2t+5}{2t+N+5}c_Nd_{2N})\\
 &=a_0b_0+(c_0d_0+\frac{4t+7}{4t+9}c_2b_1+\cdots+\frac{4t+9-M_{2k}}{4t+9}c_{M_{2k}}d_{\frac{M_{2k}}{2}}),\vspace{1mm}\\
 &\frac{2t+5}{2t+6}c_1d_1+\frac{2t+5}{2t+7}c_2d_3+\cdots+\frac{2t+5}{2t+N+5}c_Nd_{2N-1}\\
 &=\frac{4t+9}{4t+10}c_1d_0+\frac{4t+7}{4t+10}c_3d_1+\cdots+\frac{4t+10-L_{2k+1}}{4t+10}c_{L_{2k+1}}d_{\frac{L_{2k+1}-1}{2}}.
 \end{cases}
\end{align}
Subtracting the first equation of (\ref{9.3}) from the first equation of (\ref{9.4}), and subtracting the first equation of (\ref{9.4}) from the first equation of (\ref{9.5}) yields that
\begin{align*}
\begin{split}
&\left(\frac{2t+3}{2t+4}-\frac{2t+1}{2t+2}\right)c_1d_2+\cdots+\left(\frac{2t+3}{2t+N+3}-\frac{2t+1}{2t+N+1}\right)c_Nd_{2N}\\
&=\frac{8}{(4t+1)(4t+5)}c_2d_1+\frac{16}{(4t+5)(4t+1)}c_4d_2+\cdots+\frac{4M_{2k}}{(4t+5)(4t+1)}c_{M_{2k}}d_{\frac{M_{2k}}{2}};
 \end{split}
\end{align*}
\begin{align*}
\begin{split}
&\left(\frac{2t+5}{2t+6}-\frac{2t+3}{2t+4}\right)c_1d_2+\cdots+\left(\frac{2t+5}{2t+N+5}-\frac{2t+3}{2t+N+3}\right)c_Nd_{2N}\\
&=\frac{8}{(4t+5)(4t+9)}c_2d_1+\frac{16}{(4t+9)(4t+5)}c_4d_2+\cdots+\frac{4M_{2k}}{(4t+9)(4t+5)}c_{M_{2k}}d_{\frac{M_{2k}}{2}}.
 \end{split}
\end{align*}
This implies that
\begin{align*}
\begin{split}
&(4t+1)\left[\left(\frac{2t+3}{2t+4}-\frac{2t+1}{2t+2}\right)c_1d_2+\cdots+\left(\frac{2t+3}{2t+N+3}-\frac{2t+1}{2t+N+1}\right)c_Nd_{2N}\right]\\
&=(4t+9)\left[\left(\frac{2t+5}{2t+6}-\frac{2t+3}{2t+4}\right)c_1d_2+\cdots+\left(\frac{2t+5}{2t+N+5}-\frac{2t+3}{2t+N+3}\right)c_Nd_{2N}\right].
 \end{split}
\end{align*}
This equation is equals to
\begin{align*}
\begin{split}
&\left[(4t+9)\left(\frac{2t+5}{2t+6}-\frac{2t+3}{2t+4}\right)-(4t+1)\left(\frac{2t+3}{2t+4}-\frac{2t+1}{2t+2}\right)\right]c_1d_2+\cdots\\
&\ \ +\bigg[(4t+9)\Big(\frac{2t+5}{2t+N+5}-\frac{2t+3}{2t+N+3}\Big)\\
&\ \ -(4t+1)\Big(\frac{2t+3}{2t+N+3}-\frac{2t+1}{2t+N+1}\Big)\bigg]c_Nd_{2N}=0.
\end{split}
\end{align*}
Elementary calculations give us that the coefficient of $c_jd_{2j}$ is
\begin{align*}
\begin{split}
&(4t+9)\left(\frac{2t+5}{2t+j+5}-\frac{2t+3}{2t+j+3}\right)-(4t+1)\left(\frac{2t+3}{2t+j+3}-\frac{2t+1}{2t+j+1}\right)\\
&=\frac{8j(2j+1)}{(2t+j+1)(2t+j+3)(2t+j+5)}\\
&=8j(2j+1)\left(\frac{1}{8(2t+j+1)}-\frac{1}{4(2t+j+3)+\frac{1}{(2t+j+5)}}\right).
\end{split}
\end{align*}
Letting $t=N,N+1,\cdots,2N-1,\cdots,2N+3$, we will get a homogeneous system of linear equations of $c_1d_2,c_2d_4,\cdots,c_Nd_{2N}$, and the coefficient matrix of the system is
\begin{align}\label{9.7}
\left( \begin{array} {c c c c c c}
{{\frac{1}{(2N+2)(2N+4)(2N+6)}}} & {{\frac{1}{(2N+3)(2N+5)(2N+7)}}} & {{\cdots}} & {{\frac{1}{(2N+N+1)(2N+N+3)(2N+N+5)}}}\\ 
{{\frac{1}{(2N+4)(2N+6)(2N+8)}}} & {{\frac{1}{(2N+5)(2N+7)(2N+9)}}} & {{\cdots}} & {{\frac{1}{(2N+N+3)(2N+N+5)(2N+N+7)}}}\\ 
{{\cdots}} & {{\cdots}} & {{\cdots}} & {{\cdots}}\\ 
{{\frac{1}{4N(4N+2)(4N+4)}}} & {{\frac{1}{(4N+1)(4N+3)(4N+5)}}} & {{\cdots}} & {{\frac{1}{(5N-1)(5N+1)(5N+3)}}}\\ 
{{\cdots}} & {{\cdots}} & {{\cdots}} & {{\cdots}}\\ 
{{\frac{1}{(4N+8)(4N+10)(4N+12)}}} & {{\frac{1}{(4N+9)(4N+11)(4N+13)}}} & {{\cdots}} & {{\frac{1}{(5N+7)(5N+9)(5N+11)}}}\\ 
\end{array} \right),
\end{align}
denoted these $N+4$ row vectors by $\alpha_1,\alpha_2,\cdots,\alpha_{N+4}$, thus, we have
\begin{align*}
  \left(\begin{array}{c}
          \alpha_1 \\
          \alpha_2 \\
          \alpha_3 \\
          \cdots \\
          \alpha_{N+4} 
        \end{array}
  \right)
  =
  &\left(
  \begin{array}{ccccccc}
    \frac{1}{2N+2} & \frac{1}{2N+3} & \frac{1}{2N+4} & \cdots & \frac{1}{2N+N+1} & \frac{1}{2N+N+3} & \frac{1}{2N+N+5} \\
    \frac{1}{2N+4} & \frac{1}{2N+5} & \frac{1}{2N+6} & \cdots & \frac{1}{2N+N+3} & \frac{1}{2N+N+5} & \frac{1}{2N+N+7} \\
    \frac{1}{2N+6} & \frac{1}{2N+7} & \frac{1}{2N+8} & \cdots & \frac{1}{2N+N+5} & \frac{1}{2N+N+7} & \frac{1}{2N+N+9} \\
    \cdots & \cdots & \cdots &   & \cdots & \cdots & \cdots \\
    \frac{1}{4N+8} & \frac{1}{4N+12} & \frac{1}{4N+16} & \cdots & \frac{1}{5N+7} & \frac{1}{5N+9} & \frac{1}{5N+11} 
  \end{array}
  \right)\\
  &\left(\begin{array}{ccccc}
          \frac{1}{8} &   &   &   &   \\
           0 & \frac{1}{8} &   &   &   \\
          -\frac{1}{4} & 0 & \frac{1}{8} &   &   \\
          0 & -\frac{1}{4} & 0 & \ddots & \frac{1}{8} \\
          \frac{1}{8} & 0 & -\frac{1}{4} & \ddots & 0 \\
            & \frac{1}{8} & 0 & \ddots & -\frac{1}{4} \\
            &   & \frac{1}{8} & \ddots & 0 \\
            &   &   &   & \frac{1}{8} 
        \end{array}
  \right)\\
  &:=AB.
\end{align*}
Note that $A$ is a Hilbert matrix, the rank of $A$ is $N+4$, obviously, the rank of $B$ is $N$. By the inequality of rank about matrix, we have
$$r(AB)\geq r(A)+r(B)-(N+4)=N+4+N-(N+4)=N.$$

Consequently, we may select $N$ linearly independent row vectors from the set $\alpha_1,\alpha_2,\cdots,\alpha_{N+4}$, denoted as $\alpha_ {i_1},\alpha_{i_2},\cdots,\alpha_{i_N}$. This selection induces a homogeneous linear system of $N$ equations in the variables $c_1d_2,c_2d_4,\cdots,c_Nd_{2N}$. Since the coefficient matrix of this system has full rank $N$, the system admits only the trivial solution. Therefore, we conclude that:  
$$c_1d_2=c_2d_4=\cdots=c_Nd_ {2N}=0.$$

Similarly, subtracting the second equation of (\ref{9.3}) from the second equation of (\ref{9.4}), and subtracting the second equation of (\ref{9.4}) from the second equation of (\ref{9.5}) yields that
\begin{align*}
\begin{split}
&\left(\frac{2t+3}{2t+4}-\frac{2t+1}{2t+2}\right)c_1d_1+\cdots+\left(\frac{2t+3}{2t+N+3}-\frac{2t+1}{2t+N+1}\right)c_Nd_{2N-1}\\
&=\frac{4}{(4t+2)(4t+6)}c_1d_0+\frac{12}{(4t+2)(4t+6)}c_3d_1+\cdots+\frac{4L_{2k+1}}{(4t+2)(4t+6)}c_{L_{2k+1}}d_{\frac{L_{2k+1}}{2}};
 \end{split}
\end{align*}
\begin{align*}
\begin{split}
&\left(\frac{2t+5}{2t+6}-\frac{2t+3}{2t+4}\right)c_1d_1+\cdots+\left(\frac{2t+5}{2t+N+5}-\frac{2t+3}{2t+N+3}\right)c_Nd_{2N-1}\\
&=\frac{4}{(4t+6)(4t+10)}c_1d_0+\frac{12}{(4t+6)(4t+10)}c_3d_1+\cdots+\frac{4L_{2k+1}}{(4t+6)(4t+10)}c_{L_{2k+1}}d_{\frac{L_{2k+1}-1}{2}}.
 \end{split}
\end{align*}
This implies that
\begin{align*}
\begin{split}
&(4t+2)\left[\left(\frac{2t+3}{2t+4}-\frac{2t+1}{2t+2}\right)c_1d_1+\cdots+\left(\frac{2t+3}{2t+N+3}-\frac{2t+1}{2t+N+1}\right)c_Nd_{2N-1}\right]\\
&=(4t+10)\left[\left(\frac{2t+5}{2t+6}-\frac{2t+3}{2t+4}\right)c_1d_1+\cdots+\left(\frac{2t+5}{2t+N+5}-\frac{2t+3}{2t+N+3}\right)c_Nd_{2N-1}\right].
 \end{split}
\end{align*}
This equation is equals to
\begin{align}\label{9.8}
\begin{split}
&\left[(4t+10)\left(\frac{2t+5}{2t+6}-\frac{2t+3}{2t+4}\right)-(4t+2)\left(\frac{2t+3}{2t+4}-\frac{2t+1}{2t+2}\right)\right]c_1d_2+\cdots\vspace{2mm}\\
&\ \ +\bigg[(4t+10)\Big(\frac{2t+5}{2t+N+5}-\frac{2t+3}{2t+N+3}\Big)\\
&\ \ -(4t+2)\Big(\frac{2t+3}{2t+N+3}-\frac{2t+1}{2t+N+1}\Big)\bigg]c_Nd_{2N-1}=0.
\end{split}
\end{align}
Elementary calculations give us that the coefficient of $c_jd_{2j-1}$ is
\begin{align*}
\begin{split}
&(4t+10)\left(\frac{2t+5}{2t+j+5}-\frac{2t+3}{2t+j+3}\right)-(4t+2)\left(\frac{2t+3}{2t+j+3}-\frac{2t+1}{2t+j+1}\right)\\
&=\frac{16j^2}{(2t+j+1)(2t+j+3)(2t+j+5)}\\
&=16j^2\left(\frac{1}{8(2t+j+1)}-\frac{1}{4(2t+j+3)+\frac{1}{(2t+j+5)}}\right).
\end{split}
\end{align*}

Consider $t=N,N+1,\cdots,2N-1,\cdots,2N+3$. This generates a homogeneous linear system in the variables $c_1d_1,c_2d_3,\cdots,c_Nd_{2N-1}$, with its coefficient matrix identical to that of system (\ref{9.7}). Solving this system yields:  
$$c_1d_1=c_2d_3=\cdots=c_Nd_ {2N-1}=0.$$
Since the analytic polynomial is of degree $N$, we have $c_N\neq 0$. Consequently,
$$d_{2N}=d_{2N-1}=0.$$

{\textbf{Substep 1.2}} Suppose $2t>N$. Letting $s=t,k=2t+1;s=t+1,k=2t+3;s=t+2,k=2t+5$ in (\ref{9.1}) and (\ref{9.2}), we get the followig equations:
\begin{align}\label{9.9}
 \begin{cases}
 &\frac{2t+1}{2t+2}c_1d_0+\cdots+\frac{2t+1}{2t+N+1}c_Nd_{2N-2}=\frac{4t+1}{4t+3}c_02d_0+\cdots+\frac{4t+3-M_{2k}}{4t+3}c_{M_{2k}}d_{\frac{M_{2k}}{2}-1},\vspace{2mm}\\
 &\frac{2t+1}{2t+3}c_2d_1+\frac{2t+1}{2t+4}c_3d_3+\cdots+\frac{2t+1}{2t+N+1}c_Nd_{2N-3}\vspace{2mm}\\
 &=\frac{4t+1}{4t+4}c_3d_0+\frac{4t-1}{4t+4}c_5d_1+\cdots+\frac{4t+4-L_{2k+1}}{4t+4}c_{L_{2k+1}}d_{\frac{L_{2k+1}-3}{2}};
 \end{cases}
\end{align}
\begin{align}\label{9.10}
 \begin{cases}
 &\frac{2t+3}{2t+4}c_1d_0+\cdots+\frac{2t+3}{2t+N+3}c_Nd_{2N-2}=\frac{4t+5}{4t+7}c_2d_0+\cdots+\frac{4t+7-M_{2k}}{4t+7}c_{M_{2k}}d_{\frac{M_{2k}}{2}-1},\vspace{2mm}\\
 &\frac{2t+3}{2t+5}c_1d_1+\frac{2t+3}{2t+5}c_2d_3+\cdots+\frac{2t+3}{2t+N+3}c_Nd_{2N-1}\vspace{2mm}\\
 &=\frac{4t+5}{4t+6}c_1d_0+\frac{4t+3}{4t+6}c_3d_1+\cdots+\frac{4t+6-L_{2k+1}}{4t+6}c_{L_{2k+1}}d_{\frac{L_{2k+1}-3}{2}};
 \end{cases}
\end{align}
\begin{align}\label{9.11}
 \begin{cases}
 &\frac{2t+5}{2t+6}c_1d_0+\cdots+\frac{2t+5}{2t+N+5}c_Nd_{2N-2})=\frac{4t+9}{4t+11}c_2d_0+\cdots+\frac{4t+11-M_{2k}}{4t+11}c_{M_{2k}}d_{\frac{M_{2k}}{2}-1},\vspace{2mm}\\
 &\frac{2t+5}{2t+7}c_2d_1+\frac{2t+5}{2t+8}c_3d_3+\cdots+\frac{2t+5}{2t+N+5}c_Nd_{2N-3}\\
 &=\frac{4t+9}{4t+12}c_3d_0+\frac{4t+7}{4t+12}c_5d_1+\cdots+\frac{4t+10-L_{2k+1}}{4t+10}c_{L_{2k+1}}d_{\frac{L_{2k+1}-3}{2}}.
 \end{cases}
\end{align}
Subtracting the first equation of (\ref{9.9}) from the first equation of (\ref{9.10}), and subtracting the first equation of (\ref{9.10}) from the first equation of (\ref{9.11}) yields that
\begin{align*}
\begin{split}
&\left(\frac{2t+3}{2t+4}-\frac{2t+1}{2t+2}\right)c_1d_0+\cdots+\left(\frac{2t+3}{2t+N+3}-\frac{2t+1}{2t+N+1}\right)c_Nd_{2N-2}\\
&=\frac{8}{(4t+3)(4t+7)}c_2d_1+\frac{16}{(4t+3)(4t+7)}c_4d_2+\cdots+\frac{4M_{2k}}{(4t+3)(4t+7)}c_{M_{2k}}d_{\frac{M_{2k}-2}{2}};
 \end{split}
\end{align*}
\begin{align*}
\begin{split}
&\left(\frac{2t+5}{2t+6}-\frac{2t+3}{2t+4}\right)c_1d_0+\cdots+\left(\frac{2t+5}{2t+N+5}-\frac{2t+3}{2t+N+3}\right)c_Nd_{2N-2}\\
&=\frac{8}{(4t+7)(4t+11)}c_2d_1+\frac{16}{(4t+7)(4t+11)}c_4d_2+\cdots+\frac{4M_{2k}}{(4t+7)(4t+11)}c_{M_{2k}}d_{\frac{M_{2k}-2}{2}}.
 \end{split}
\end{align*}
This implies that
\begin{align*}
\begin{split}
&(4t+3)\left[\left(\frac{2t+3}{2t+4}-\frac{2t+1}{2t+2}\right)c_1d_0+\cdots+\left(\frac{2t+3}{2t+N+3}-\frac{2t+1}{2t+N+1}\right)c_Nd_{2N-2}\right]\\
&=(4t+11)\left[\left(\frac{2t+5}{2t+6}-\frac{2t+3}{2t+4}\right)c_1d_0+\cdots+\left(\frac{2t+5}{2t+N+5}-\frac{2t+3}{2t+N+3}\right)c_Nd_{2N-2}\right].
 \end{split}
\end{align*}
This equation is equals to
\begin{align*}
\begin{split}
&\left[(4t+11)\left(\frac{2t+5}{2t+6}-\frac{2t+3}{2t+4}\right)-(4t+3)\left(\frac{2t+3}{2t+4}-\frac{2t+1}{2t+2}\right)\right]c_1d_0+\cdots\\
&+\left[(4t+11)\left(\frac{2t+5}{2t+N+5}-\frac{2t+3}{2t+N+3}\right)-(4t+3)\left(\frac{2t+3}{2t+N+3}-\frac{2t+1}{2t+N+1}\right)\right]c_Nd_{2N-2}=0.
\end{split}
\end{align*}
By the similar method as equation (\ref{9.8}), we have
$$c_1d_0=c_2d_2=\cdots=c_Nd_{2N-2}=0.$$

Similarly, subtracting the second equation of (\ref{9.9}) from the second equation of (\ref{9.10}), and subtracting the second equation of (\ref{9.10}) from the second equation of (\ref{9.11}) yields that
\begin{align*}
\begin{split}
&\left[(4t+12)\left(\frac{2t+5}{2t+7}-\frac{2t+3}{2t+5}\right)-(4t+4)\left(\frac{2t+3}{2t+5}-\frac{2t+1}{2t+3}\right)\right]c_2d_1+\cdots\\
&+\left[(4t+12)\left(\frac{2t+5}{2t+N+5}-\frac{2t+3}{2t+N+3}\right)-(4t+4)\left(\frac{2t+3}{2t+N+3}-\frac{2t+1}{2t+N+1}\right)\right]c_Nd_{2N-3}=0.
\end{split}
\end{align*}
This implies  
$$c_2d_1=c_3d_4=\cdots=c_Nd_{2N-3}=0.$$ 
Given that the analytic polynomial is of degree $N$, we have $c_N\neq 0$. Consequently, it follows that  
$$d_{2N-2}=d_{2N-3}=0.$$ 

Iterating this procedure by successively assigning:  
$s=t+N,k=2t+N;s=t+N+1,k=2t+N+2;s=t+N+2,k=2t+N+4,$ 
to equations (\ref{9.1}) and (\ref{9.2}), we ultimately derive  
$$d_0=d_1=\cdots=d_{2N}=0.$$  

\textbf{Step 2.}
By setting $s=t+N,k=2t+N;s=t+N+1,k=2t+N+2;s=t+N+2,k=2t+N+4$ in (\ref{9.1}) and (\ref{9.2}), we get the following equations:
\begin{align}\label{9.12}
\begin{cases}
  &a_Nb_0+a_{N-1}b_2+\cdots+a_0b_{2N}+(c_0d_{2N}+\frac{2(t+N)+1}{2(t+N)+2}c_1d_{2N+2}\vspace{2mm}\\
  &+\frac{2(t+N)+1}{2(t+N)+3}c_2d_{2N+4}+\cdots+\frac{2(t+N)+1}{2(t+N)+N+1}c_Nd_{4N}) \vspace{2mm}\\
  &=a_0b_N+a_2b_{N-1}+\cdots+a_{2N}d_0+(c_0d_N+\frac{2(2t+N)-1}{2(2t+N)+1}c_2d_{N+1}\vspace{2mm}\\
  &+\frac{2(2t+N)-3}{2(2t+N)+1}c_4d_{N+2}+\cdots+\frac{2(2t+N)-M_{2k}+1}{2(2t+N)+1}c_{M_{2k}}d_{N+\frac{M_{2k}}{2}}),\vspace{2mm}\\
  &a_{N-1}b_1+a_{N-2}b_3+\cdots+a_0b_{2N-1}+(c_0d_{2N-1}+\frac{2(t+N)+1}{2(t+N)+2}c_1d_{2N+1}\vspace{2mm}\\
  &+\frac{2(t+N)+1}{2(t+N)+3}c_2d_{2N+3}+\cdots+\frac{2(t+N)+1}{2(t+N)+N+1}c_Nd_{4N-1})\vspace{2mm} \\
  &=a_0b_{2N-1}+a_3b_{2N-2}+\cdots+a_{2N-1}d_0+(c_1d_N+\frac{2(2t+N)-1}{2(2t+N)+2}c_3d_{N+1}\vspace{2mm}\\
  &+\frac{2(2t+N)-3}{2(2t+N)+2}c_5d_{N+2}+\cdots+\frac{2(2t+N)-L_{2k+1}+2}{2(2t+N)+2}c_{L_{2k+1}}d_{N+\frac{L_{2k+1}+1}{2}-1});
\end{cases}
\end{align}
\begin{align}\label{9.13}
\begin{cases}
  &a_Nb_0+a_{N-1}b_2+\cdots+a_0b_{2N}+(c_0d_{2N}+\frac{2(t+N+1)+1}{2(t+N+1)+2}c_1d_{2N+2}\vspace{2mm}\\
  &+\frac{2(t+N+1)+1}{2(t+N+1)+3}c_2d_{2N+4}+\cdots+\frac{2(t+N+1)+1}{2(t+N+1)+N+1}c_Nd_{4N}) \vspace{2mm}\\
  &=a_0b_N+a_2b_{N-1}+\cdots+a_{2N}d_0+(c_0d_N+\frac{2(2t+N+2)-1}{2(2t+N+2)+1}c_2d_{N+1}\vspace{2mm}\\
  &+\frac{2(2t+N+2)-3}{2(2t+N+2)+1}c_4d_{N+2}+\cdots+\frac{2(2t+N+2)-M_{2k}+1}{2(2t+N+2)+1}c_{M_{2k}}d_{N+\frac{M_{2k}}{2}}),\vspace{2mm}\\
  &a_{N-1}b_1+a_{N-2}b_3+\cdots+a_0b_{2N-1}+(c_0d_{2N-1}+\frac{2(t+N+1)+1}{2(t+N+1)+2}c_1d_{2N+1}\vspace{2mm}\\
  &+\frac{2(t+N+1)+1}{2(t+N+1)+3}c_2d_{2N+3}+\cdots+\frac{2(t+N+1)+1}{2(t+N+1)+N+1}c_Nd_{4N-1}) \vspace{2mm}\\
  &=a_0b_{2N-1}+a_3b_{2N-2}+\cdots+a_{2N-1}d_0+(c_1d_N+\frac{2(2t+N+2)-1}{2(2t+N+2)+2}c_3d_{N+1}\vspace{2mm}\\
  &+\frac{2(2t+N+2)-3}{2(2t+N+2)+2}c_5d_{N+2}+\cdots+\frac{2(2t+N+2)-L_{2k+1}+2}{2(2t+N+2)+2}c_{L_{2k+1}}d_{N+\frac{L_{2k+1}+1}{2}-1});
\end{cases}
\end{align}
\begin{align}\label{9.14}
\begin{cases}
  &a_Nb_0+a_{N-1}b_2+\cdots+a_0b_{2N}+(c_0d_{2N}+\frac{2(t+N+2)+1}{2(t+N+2)+2}c_1d_{2N+2}\\
  &+\frac{2(t+N+2)+1}{2(t+N+2)+3}c_2d_{2N+4}+\cdots+\frac{2(t+N+2)+1}{2(t+N+2)+N+1}c_Nd_{4N}) \\
  &=a_0b_N+a_2b_{N-1}+\cdots+a_{2N}d_0+(c_0d_N+\frac{2(2t+N+4)-1}{2(2t+N+4)+1}c_2d_{N+1}\\
  &+\frac{2(2t+N+4)-3}{2(2t+N+4)+1}c_4d_{N+2}+\cdots+\frac{2(2t+N+4)-M_{2k}+1}{2(2t+N+4)+1}c_{M_{2k}}d_{N+\frac{M_{2k}}{2}}),\vspace{2mm}\\
  &a_{N-1}b_1+a_{N-2}b_3+\cdots+a_0b_{2N-1}+(c_0d_{2N-1}+\frac{2(t+N+2)+1}{2(t+N+2)+2}c_1d_{2N+1}\\
  &+\frac{2(t+N+2)+1}{2(t+N+2)+3}c_2d_{2N+3}+\cdots+\frac{2(t+N+2)+1}{2(t+N+2)+N+1}c_Nd_{4N-1}) \\
  &=a_0b_{2N-1}+a_3b_{2N-2}+\cdots+a_{2N-1}d_0+(c_1d_N+\frac{2(2t+N+4)-1}{2(2t+N+4)+2}c_3d_{N+1}\\
  &+\frac{2(2t+N+4)-3}{2(2t+N+4)+2}c_5d_{N+2}+\cdots+\frac{2(2t+N+4)-L_{2k+1}+2}{2(2t+N+4)+2}c_{L_{2k+1}}d_{N+\frac{L_{2k+1}+1}{2}-1}).
\end{cases}
\end{align}
Subtracting the first equation of (\ref{9.12}) from the first equation of (\ref{9.13}), and subtracting the first equation of (\ref{9.13}) from the first equation of (\ref{9.14}) yields that
\begin{align*}
\begin{split}
&\left(\frac{2(t+N+1)+1}{2(t+N+1)+2}-\frac{2(t+N)+1}{2(t+N)+2}\right)c_1d_{2N+2}+\cdots
+\Big(\frac{2(t+N+1)+1}{2(t+N+1)+N+1}\\
&\ \ \ \ \ \ \ \ \ \ \ \ \ \ -\frac{2(t+N)+1}{2(t+N)+N+1}\Big)c_Nd_{4N}\\ &=\frac{8c_2d_{N+1}+16c_4d_{N+2}+\cdots+4M_{2k}c_{M_{2k}}d_{N+\frac{M_{2k}-2}{2}}}{[2(2t+N+2)+1][2(2t+N)+1]};
\end{split}
\end{align*}
\begin{align*}
\begin{split}
&\left(\frac{2(t+N+2)+1}{2(t+N+2)+2}-\frac{2(t+N+1)+1}{2(t+N+1)+2}\right)c_1d_{2N+2}+\cdots
+\Big(\frac{2(t+N+2)+1}{2(t+N+2)+N+1}\\
&\ \ \ \ \ \ \ \ \ \ \ \ \ \ -\frac{2(t+N+1)+1}{2(t+N+1)+N+1}\Big)c_Nd_{4N}\\
&=\frac{8c_2d_{N+1}+16c_4d_{N+2}+\cdots+4M_{2k}c_{M_{2k}}d_{N+\frac{M_{2k}-2}{2}}}{[2(2t+N+4)+1][2(2t+N+2)+1]}.
 \end{split}
\end{align*}

By the similar method as equation (\ref{9.8}), we have
$$c_1d_{2N+2}=c_2d_{2N+4}=\cdots=c_Nd_{4N}=0.$$

Similarly, subtracting the second equation of (\ref{9.12}) from the second equation of (\ref{9.13}), and subtracting the second equation of (\ref{9.13}) from the second equation of (\ref{9.14}) yields that
$$c_1d_{2N+1}=c_2d_{2N+3}=\cdots=c_Nd_{4N-1}=0,$$
as the order of analytic polynomial is $N$, then $c_N\neq 0$, thus we have $d_{4N}=d_{4N-1}=0.$

Keep doing this process, until let $s=t+N,k=2t+2N;s=t+N+1,k=2t+2N+2;s=t+N+2,k=2t+2N+4,$ respectively, in (\ref{9.1}) and (\ref{9.2}), we finally conclude that
$$d_{2N+1}=d_{2N+2}=\cdots=d_{4N}=0.$$

Finally, we will get $d_j=0$, for every $j\geq 0$, this means that $h=0$, thus $\varphi=\psi$.
\end{proof}
\begin{rem}
Theorem \ref{c} was also obtained using a different mathod in \cite{[11]}.
\end{rem}

\end{document}